\documentclass[11pt]{article}
\usepackage{amsfonts}
\usepackage{graphicx}
\usepackage{subfig}
\usepackage{latexsym}
\usepackage{tabularx,array}
\usepackage{url}
\usepackage{comment}
\usepackage{amsmath,amsthm, dsfont, bm,bbm,xcolor}
\usepackage{titlesec}
\usepackage{mathtools}
\mathtoolsset{showonlyrefs}

\newcommand{\R}{\mathbb{R}}

\newcommand{\E}{\mathbb{E}}
\newcommand{\N}{\mathbb{N}}

\renewcommand{\P}{\mathbb P}

\newcommand{\Pb}{\mathbb{P}}

\newcommand{\al}{\alpha}
\newcommand{\be}{\beta}

\newcommand{\ga}{\gamma}

\newcommand{\si}{\sigma}
\newcommand{\Si}{\Sigma}

\newcommand{\de}{\delta}

\newcommand{\f}{\mathcal F}

\titleformat{\paragraph}
{\normalfont\normalsize\bfseries}{\theparagraph}{1em}{}
\titlespacing*{\paragraph}
{0pt}{3.25ex plus 1ex minus .2ex}{1.5ex plus .2ex}

\newcommand{\Indi}[1]{\mathbbm{1}_{#1}}

\newcommand{\rev}{\color{black}}

\newcommand{\proba}{(\Omega ,\mathcal{F},(\f_t)_{t\geq0},\P)}
\newcommand{\prob}{(\Omega ,\mathcal{F},\P)}
\newcommand{\probp}{(\Omega' ,\mathcal{F}',\P')}

\newcommand{\toop}{\stackrel{\P}{\longrightarrow}}
\newcommand{\schw}{\stackrel{\raisebox{-1pt}{\textup{\tiny d}}}{\longrightarrow}}
\newcommand{\fidi}{\stackrel{\mathcal{G}-\textup{\tiny st}}{\longrightarrow}}

\newcommand{\lan}{\langle}
\newcommand{\ran}{\rangle}

\newcommand{\bee}{\begin{equation}}
\newcommand{\eee}{\end{equation}}
\newcommand{\beea}{\begin{array}}
\newcommand{\eeea}{\end{array}}

\renewcommand{\theequation}{\arabic{section}.\arabic{equation}}

\theoremstyle{plain}
\newtheorem{prop}{Proposition}[section]

\newtheorem{theo}[prop]{Theorem}

\newtheorem{lemma}[prop]{Lemma}

\theoremstyle{definition}

\newtheorem{rem}[prop]{Remark}

\newcommand{\revch}{\color{black}}

\begin{document}

\title{Quantitative and stable limits of high-frequency statistics of L\'evy processes: a Stein's method approach
\thanks{The authors
gratefully acknowledge financial support of ERC Consolidator Grant 815703
``STAMFORD: Statistical Methods for High Dimensional Diffusions''.}} 
\author{Chiara Amorino\thanks{Department
of Mathematics, University of Luxembourg, 
E-mail: chiara.amorino@uni.lu.}  \and
Arturo  Jaramillo\thanks{Center of Research in Mathematics (CIMAT), Guanajuato, 
E-mail: jagil@cimat.mx.} \and
Mark Podolskij\thanks{Department
of Mathematics, University of Luxembourg,
E-mail: mark.podolskij@uni.lu.}}

\maketitle
\begin{abstract}
We establish inequalities for assessing the distance between the distribution of errors of partially observed high-frequency statistics of multidimensional L\'evy processes and that of a mixed Gaussian random variable. Furthermore, we provide a general result guaranteeing stable  convergence. Our arguments rely on a suitable adaptation of the Stein's method perspective to the context of mixed Gaussian distributions, specifically tailored to the framework of high-frequency statistics.\\

 
\end{abstract}


\section{Introduction} \label{sec1}
\setcounter{equation}{0}
\renewcommand{\theequation}{\thesection.\arabic{equation}}
The aim of this paper is to establish a collection of mathematical techniques, allowing one to describe  mixed Gaussian convergences and tightly assess the associated marginal distance, within the context of high-frequency statistics of L\'evy processes. More precisely, given a  $d$-dimensional L\'evy process  $\bm{X}=(\bm{X}_t)_{t\geq 0}$  of the form $\bm{X}_t =(X_t^{(1)},\dots,X_t^{(d)})$, defined on a filtered probability space $\proba$,  we consider the high-frequency statistics

\begin{align}\label{eq:statfirstapp G}
G_{t}^{(n)}
  &:=\sum_{i=1}^{\lfloor nt \rfloor} g\left( \bm{X}_{\frac{i-1}{n}} ,
\mathcal{{I}}_{i,n}\right),
\end{align}
for $t\geq 0$, $g:\R^{d(m + 2)}\rightarrow\R$ and 
\begin{align}\label{eq:Indeffirst}
\mathcal{I}_{i,n}
  :=\left( a_n(\bm{X}_{\frac{i}{n}}-\bm{X}_{\frac{i-1}{n}}) , a_n(\bm{X}_{\frac{i + 1}{n}}-\bm{X}_{\frac{i}{n}}) , \dots , a_n( \bm{X}_{\frac{i + m}{n}}-\bm{X}_{\frac{i + m-1}{n}})\right)
\end{align}
for $m\geq 0$ and some suitable normalization constant $a_n$.  The monitoring of any given time-evolving system is typically composed of an entirely observable component which is progressively revealed and a random perturbation. Such high-frequency statistics are very popular in e.g. mathematical finance, where $\bm{X}$ represents a price process. For more specific examples of functionals of type \eqref{eq:statfirstapp G}, we refer to estimation of local times of $\al$-stable L\'evy processes \cite{AJP22}, (weighted) characteristic functions \cite{MR3495691} and (weighted) empirical distributed functions \cite{MR3226166}.  
We recommend the monograph \cite{JP12} for a comprehensive study of high-frequency statistics of It\^o semimartingales.  
From the martingale theory perspective the natural centring 
of $g( \bm{X}_{\frac{i-1}{n}},\mathcal{I}_{i,n})$ is given by conditional expectation 
$\E[g( \bm{X}_{\frac{i-1}{n}},\mathcal{I}_{i,n})\ |\ \mathcal{F}_{(i-1)/n}]$. Thus, the quantity of interest becomes the statistic
\begin{align*}
E_{t}^{(n)}
  &:=\sum_{i=1}^{\lfloor nt \rfloor} \left(g ( \bm{X}_{\frac{i-1}{n}} ,
\mathcal{I}_{i,n} ) - \E[g ( \bm{X}_{\frac{i-1}{n}} ,
\mathcal{I}_{i,n})\ |\ \mathcal{F}_{\frac{i-1}{n}}] \right).
\end{align*}
The description of the asymptotic behavior of $E_{t}^{(n)}$ is the main object of our study, with emphasis on 
determining bounds for rate of convergence of a suitable probabilistic distance between the underlying marginal distributions and their limit.

In order to tackle this problem, we extend the ideas of Stein's method, developed by Charles Stein in his fundamental work \cite{ChStein}, into a framework  designed ad-hoc to the problem described above, which will allow us to both determine asymptotic mixed normality and assess the associated rate of convergence. It is worth observing that a universal Stein's method for mixed Gaussian limits constitutes a virtually unexplored topic in the current literature. This is due to the flexibility of choosing a random variance in a mixed Gaussian limit approximation, which shatters the simplicity of its distribution and demands very precise knowledge of the sequence of distributions under consideration, information which is not available in most cases. Besides this technical challenge, a handful of special instances have been treated in the literature, among which we highlight the work on additive functionals of the Brownian motion by Papanicolau-Stroock-Varadhan Theorem (see e.g. \cite[Theorem XIII-(2.6) and Proposition XIII-(2.8)]{RY}), the subsequent extensions to the fractional Brownian motion case by Hu, Jaramillo, Nualart, Nourdin, Peccati and  Xu in \cite{HuNuXu,JaNoNuPE,NuXu}, and the work of Nourdin et.al. \cite{NoNu,NoNuPe}, where mixed Gaussian limits are studied in the context of variables in the Wiener chaos. In the continuous semimartingale setting Edgeworth expansions associated with mixed Gaussian limits and their applications have been investigated in \cite{PVY17,PVY20, PY16, Y13}, with Malliavin calculus and martingale techniques being the main mathematical tools. Later the theoretical results have been extended to the framework of Skorohod integrals
in \cite{NY19}. 

Despite the aforementioned studies, 
a complete understanding of the mixed Gaussian phenomenology is far from being complete, as the entirety of the available results are valid only in settings where the underlying randomness emerges from Gaussian processes. The goal of this paper is to 
contribute to the completion of this theory. In particular, we provide a general {\rev criterion} 
for stable convergence towards a mixed normal limit  in Theorem \ref{stein}, which  relies on an adaptation of Stein's lemma and does not require the random variables to be 
embedded in a Gaussian space. The reach of the technique is illustrated by some remarks and example of applications, such as its implementation in the study of $E_{t}^{(n)}$. Our second main statement is Theorem \ref{th1}, which studies the rates of convergence associated with 
asymptotic mixed normality of the statistic $E_{t}^{(n)}$.  This result builds upon the classical ideas from the Lindeberg method \cite{Lindeberg} and the $K$-function approach \cite{CheGoShao}, although its implementation requires the incorporation of new ideas that could deal with the stochastic variance of the limit. 

The remainder of this paper is organized as follows. In Section \ref{sec1.1}, we introduce the basic setting under which we state our results and present some particular instances that illustrate (and motivate) the reach of the theory developed in the manuscript.  In Section \ref{sec1.3}, we present our main results, which include stable limit theorems and a quantitative evaluation of the marginal distribution of the processes under consideration. In Section \ref{s: proof stable}, we prove a general criterion for asymptotic mixed normality and the previously mentioned quantitative assessment. Additional technical results, along with their proofs, can be found in the Supplementary Material.

\subsection*{Notation}
All strictly positive constants are denoted by $C$ although they may change from line to line. We use the notation $\bm{i}=\sqrt{-1}$. We write $C^p(\R^l;\R^q)$ to denote the class of $p$ times continuously differentiable functions from $\R^l$ to $\R^q$. For $g\in C^1(\R^l;\R^q)$ we write $\nabla g$ for the derivative of $g$; for a map $(x,y)\in \R^{l}\times \R^{s} \mapsto \R^q$ we denote the partial derivatives by $\frac{\partial}{\partial x}g$ and $\frac{\partial}{\partial y}g$. When $g$ is a univariate function we denote its derivatives by $g',g'',\ldots$ or by $g^{(l)}$, $l\geq1$.   

The Euclidean (resp. sup) norm is denoted by $\|\cdot\|$ (resp. $\|\cdot\|_{\infty}$). We also define
the $\si$-algebra

\bee \label{sialg}
\mathcal{F}^{(i)} := \sigma \left({\bm{X}}_{\frac{k}{n}} - {\bm{X}}_{\frac{k-1}{n}}, {\bm{\tilde{X}}}_{s} - {\bm{\tilde{X}}}_{\frac{h-1}{n}};\, k \neq i, i+1, \ldots, i+m,\ h = i, i+1, \ldots, i+m,\ s \in \left[\tfrac{h-1}{n}, \tfrac{h}{n} \right] \right)
\eee
for $i\geq 1$, where $\tilde{X}$ is an independent but fixed copy of $X$.

\section{The setting and overview} \label{sec1.1}
\setcounter{equation}{0}
\renewcommand{\theequation}{\thesection.\arabic{equation}}

Consider a univariate L\'evy process $(X_t)_{t \ge 0}$ and let $\psi$ be its L\'evy exponent defined through
$$\E[\exp(\bm{i}\beta X_t)] = \exp(t\psi(\beta) ),  \qquad \be\in \R.$$
By the L\'evy-Khintchine formula, $\psi$ has the form
$$\psi(\beta) = \bm{i}\gamma \beta - \frac{1}{2} \sigma^2 \beta^2 + \int_{\R} \left(\exp(\bm{i} \beta x) - 1 - \bm{i} \beta x \Indi{\{|x| < 1\}}\right) \Pi (dx),$$
where $\gamma \in \R$, $\sigma \ge 0$ and $\Pi(dx)$ is a Radon measure on $[- \infty, 0) \cup (0, \infty)$ satisfying $\int_{\R} (x^2 \land 1) \Pi (dx) < \infty$.
As characteristic functions uniquely determine their underlying probability distributions, each L\'evy process is uniquely determined by the L\'evy-Khintchine triplet $(\gamma, \sigma, \Pi)$. Standard textbooks on the topic include \cite{Ber96,Kyp06,Sat13}. 

Throughout this work we consider a $d$-dimensional L\'evy process whose components $X_t^{(1)}, ... , X_t^{(d)}$ are iid and have L\'evy triplet $(\gamma, \sigma, \Pi)$. Our results will be stated in terms of a suitable smooth distance, which we define next. Let $\mathcal{H}$ be the collection of functions defined by	 
$$\mathcal{H} := \left \{ \psi: \mathbb{R} \rightarrow \mathbb{R} \mbox{ test function such that } \left \| \psi^{(l)} \right \|_\infty \le 1 \mbox{ for } l= 0,1, \dots,5 \right \}.$$
The family of functions gives raise to the smooth metric
\begin{align}\label{eq:defdistance}
d(\mu,\nu) := \sup_{\psi \in \mathcal{H}}\left|\int_{\R}\psi(x)\mu(dx) - \int_{\R}\psi(x)\nu(dx)\right|,
\end{align}
on the space of probability measures on $\R$. Observe that if $\mu_{n}$ is a sequence of probability measures satisfying $d(\mu_{n},\nu)\rightarrow0$ as $n$ tends to infinity, then the characteristic function of $\mu_{n}$ converges pointwise towards the characteristic function of  $\nu$. In particular, the metric $d$ induces a topology on the class of probability measures over $\R$ which is stronger than the topology of convergence in distribution.

Theorem \ref{th1} below reveals that the proper scaling of $E_{t}^{(n)}$ required for yielding an asymptotic non-trivial limit is  ${n}^{- \frac{1}{2}}$, reason for which the forthcoming results will be stated in terms of the following rescaling of $E_{t}^{(n)}$: 
\begin{align}\label{eq:statfirstapp}
Z_{t}^{(n)}
  &:=\frac{1}{\sqrt{n}}\sum_{i=1}^{\lfloor nt \rfloor} \left(g ( \bm{X}_{\frac{i-1}{n}} ,
\mathcal{I}_{i,n} ) - \E[g ( \bm{X}_{\frac{i-1}{n}} ,
\mathcal{I}_{i,n})\ |\ \mathcal{F}_{\frac{i-1}{n}}]\right).
\end{align}
The choice of $a_{n}$, which has been introduced in \eqref{eq:Indeffirst}, will be made in such a way that for every $\textbf{x}\in\mathbb{R}^d$ and $i\leq n$, the sequence $\text{Var}[g\left( \bm{x},\mathcal{I}_{i,n}\right)]$ converges as $n$ tends to infinity. Naturally, under suitable conditions over $g$, this convergence will automatically hold whenever $a_{n}$ is such that $a_{n}\textbf{X}_{1/n}$ converges in law. 

In addition to this, a number of adequate extra hypotheses will be needed, but before discussing them in detail we present some particular instances of our results in order to motivate the reader.\\

\noindent \textbf{A representative statement}\\
One of the main motivations of the present paper is to explore in full generality the asymptotic behavior of $Z_{t}^{(n)}$ for the case where $\textbf{X}$ is an $\alpha$-stable process with i.i.d. components and $\alpha\in(0,2)$. The case where $\textbf{X}$ is a Brownian motion is of interest as well, although it has been studied thoroughly by means of martingale techniques (see e.g. \cite{JP12}) in the case $m=0$. The martingale techniques in the case $m>0$ are not completely obvious and their level of complexity seems to be comparable to the Stein's approach presented in Theorem \ref{stein}.

As mentioned earlier, one of the methodological breakthroughs of our work is the successful implementation of Stein's method for high-frequency statistics converging towards a mixed Gaussian limit. Our main findings are stated in full generality in Theorem \ref{th1} and Theorem \ref{th2} 
below, and the following results are immediate consequences of such general statements. For this purpose, for a general L\'evy
process $\bm{X}$, we introduce {$3m+1$ independent} copies of $\bm{X}$ denoted $\Tilde{\bm{X}}^{(1 - m)}, ... ,\Tilde{\bm{X}}^{(2m)}$ and the limit as $n$ tends to infinity of the covariance of $g$. In particular, we introduce the functions $\mathfrak{g}_{n}:\R^d\rightarrow\R$ as
\begin{align}\label{eq:gnvariancedef}
\mathfrak{g}_n(\bm{x})
  &:=\sum_{j = -m}^{m} Cov \left(g (\bm{x}, a_n \Tilde{\bm{X}}_{\frac{1}{n}}^{(1)}, ... , a_n \Tilde{\bm{X}}_{\frac{1}{n}}^{(m+1)} ), g (\bm{x}, a_n \Tilde{\bm{X}}_{\frac{1}{n}}^{(j + 1)}, ... , a_n \Tilde{\bm{X}}_{\frac{1}{n}}^{(j + m+1)} ) \right).
\end{align}
In the sequel, we will assume the pointwise  existence of the limit $\lim_{n}\mathfrak{g}_{n}$, which is denoted by $\mathfrak{g}$; namely,
\begin{align}\label{eq:gvariancedef}
\mathfrak{g}(\bm{x}) &:=\lim_{n} \mathfrak{g}_n(\bm{x}).
\end{align}
Observe that in the particular case where $\bm{X}$ is {\rev an $\alpha$-stable L\'evy process, one naturally chose $a_n = n^{\frac{1}{\alpha}}$. Then,} the sequence $\mathfrak{g}_n(\bm{x})$ is a constant sequence over $n$, so in particular $\mathfrak{g}(\bm{x})=\mathfrak{g}_n(\bm{x})$ for all $n$.

\begin{theo}[Special case of Theorem \ref{th1}]\label{th1prev} 
Let $\bm{X}$ be a symmetric $\alpha$-stable process with $\alpha\in(0,2)$ and assume that $g$ is such that $\|g\|_{\infty},\|\frac{\partial g}{\partial x_{i}}\|_{\infty},\|\frac{\partial^2 g}{\partial x_{i}\partial x_{j}}\|_{\infty}<\infty,$  for all $i,j=1,\dots, d$. Then, the function  
\begin{align}\label{eq:Atdef}
 A_{s}
  := \mathfrak{g}(\bm{X}_{s})
\end{align}
is well defined. Under these conditions, for all $t>0$, there exists  a standard Brownian motion $W$ independent of $\bm{X}$, defined on an extended probability space, such that 
\begin{enumerate}
\item When $\alpha\in(1,2{)}$,
\begin{align}\label{eq:alpha12thm12}
d\left( Z_{t}^{(n)}, \int_0^{t}\sqrt{A_{s}}W(ds) \right)
  \le C n^{\frac{1}{2}-\frac{1}{\alpha}}.
\end{align}

\item When $\alpha=1$,
\bee
d\left( Z_{t}^{(n)}, \int_0^{t}\sqrt{A_{s}}W(ds)\right)
  \le \frac{C}{\sqrt{n}}\log(n).
\eee

\item When $\alpha\in(0,1)$,
\bee
d\left( Z_{t}^{(n)}, \int_0^{t}\sqrt{A_{s}}W(ds)\right)
  \le  \frac{C}{\sqrt{n}}.
\eee

\end{enumerate}
Moreover, if $g$ is symmetric over the second component, i.e. $g (\bm{x},\bm{y})=
g (\bm{x},-\bm{y})$, and $\alpha\in(1,2]$, then
\bee\label{simplifiedmainparttwogauss}
d\left( Z_{t}^{(n)}, \int_0^{t}\sqrt{A_{s}}W(ds)\right) \le \frac{C}{\sqrt{n}}.
\eee
\end{theo}

Although the above result is only a specialization of our main findings (which hold for general L\'evy processes), it illustrates quite accurately the general {spirit} of most of the results from this manuscript, as the parameter $\alpha$ (interpreted in the general framework as the shape parameter of the stable domain of attraction to which $\bm{X}$ belongs) dictates the different profiles for the bound on the rate of convergence of the distances under consideration.

\begin{rem} \rm \label{Jacodst}
We remark that the upper bound \eqref{eq:alpha12thm12} in Theorem  \ref{th1prev} does not converge to $0$ in the non-symmetric case if $\al=2$. Here we explain this phenomenon in the simple setting $m=0$ and $d=1$. 

In order to prove the asymptotic mixed normality of $Z^{(n)}$ for $m=0$ we can use Jacod's limit theorem established in   \cite{J97}. We briefly recall its statement specialised to our context. Let 
\bee
Y_t^n = \sum_{i=1}^{\lfloor nt \rfloor} \xi_i^n
\eee
be a sequence of random process with $\xi_i^n$ being $\mathcal{F}_{i/n}$-measurable, square integrable and satisfying $\E[\xi_i^n|\mathcal{F}_{(i-1)/n}]=0$. Assume that the following conditions are satisfied for any $t>0$:
\begin{align}
 &\label{condi11}\sum_{i=1}^{\lfloor nt \rfloor} \E[(\xi_i^n)^2|\mathcal{F}_{(i-1)/n}] \toop 
  \int_0^t A_s ds, \\[1.5 ex]
  &\label{condi12}\sum_{i=1}^{\lfloor nt \rfloor} \E[\xi_i^n (M_{i/n} - M_{(i-1)/n})|\mathcal{F}_{(i-1)/n}] 
  \toop 0, \\[1.5 ex]
   &\label{condi13}\sum_{i=1}^{\lfloor nt \rfloor} \E[(\xi_i^n)^2 1_{\{ |\xi_i^n|>\epsilon\}}|\mathcal{F}_{(i-1)/n}] \toop 0,
   \qquad \forall \epsilon>0,
\end{align}
where condition \eqref{condi12} holds for all square-integrable continuous martingales $(M_t)_{t \ge 0}$. Then it holds that $Y^n \stackrel{Law}{\rightarrow} Y$ on $D([0,T])$ equipped with the Skorokhod topology, where $Y_t = \int_0^t \sqrt{A_s}
W(ds)$, where $W$ is a Brownian motion independent of $\mathcal F$. It is easy to show conditions \eqref{condi11} and \eqref{condi13} for all $\al \in (0,2]$. Condition \eqref{condi12} holds for $\al \in (0,2)$, which is due to the fact that pure jump and continuous martingales are always \textit{orthogonal} in the sense of quadratic covariation (cf. \cite{AJP22}). On the other hand, when $X$ is a Brownian motion the second condition is not always true (even for $M=X$). In this case the statement  $Y^n \stackrel{Law}{\rightarrow} Y$ does not hold, which is consistent with the first result of Theorem  \ref{th1prev}
when $\al=2$.  However, in the symmetric case, condition \eqref{condi12} also holds for $\al=2$ and in this setting we do obtain a vanishing convergence rate   in Theorem  \ref{th1prev}. \qed
\end{rem}



\noindent
Using Theorem \ref{eq:Atdef} as benchmark for the type of phenomenology to expect in a more general framework, we now introduce several technical conditions which will be used throughout the paper. \\

\noindent \textbf{Conditions on $g$}\\
\noindent We will assume that $g\in C^3(\R^{d}\times \R^{d(m + 1)};\R)$  with 
$$\|g\|_{\infty},\|\frac{\partial g}{\partial x_{i}}\|_{\infty},\|\frac{\partial^2 g}{\partial x_{i}\partial x_{j}}\|_{\infty}<\infty,$$
for all $i,j=1,\dots, d$. As mentioned before, we will assume the existence of the limit as $n$ tends to infinity of the term $\text{Var}[g(\bm{x},\mathcal{I}_{i,n})]$. \\

\noindent \textbf{Conditions on $\bm{X}$}\\
We will assume that there exists a constant $\alpha\in(0,2]$ and $\kappa>0$, such that the following behavior of the tail probabilities of $\bm{X}$ holds:
\begin{align}\label{eq:Xtails}
\Pb[|X_{t}^{(1)}| \geq s]
  &\leq\kappa ts^{-\alpha}.
\end{align}
In the sequel, the collection of conditions described above will be referred to as 'hypothesis' $\bm{H_1}(\alpha)$.\ \\ 

\noindent \textbf{Optional conditions for finer results}\\
The most general versions of our results are stated only under the validity of $\bm{H_1}(\alpha)$. However, sharper bounds on the distance of the approximations discussed below are possible under additional assumptions. In particular, some instances of Theorem \ref{th1} are stated under suitable compatibility between the function $g$ and the law of $\bm{X}$, which consists of a relaxation of the condition of symmetry on the  $\alpha$-stable case. We will refer to this condition as Hypothesis $\bm{H_2}(\alpha)$:\\


{\noindent \textit{Hypothesis }$\bm{H_2}(\alpha)$\\
{\rev When $\alpha\in(1,2]$ hypothesis $\bm{H_1}(\alpha)$ holds and additionally}, for every $\bm{x}\in\R^{d}$ and $n\in\mathbb{N}$, we have that 
\begin{align*}
\E[(g(\bm{x}, \mathcal{I}_{i,n})-\E[g(\bm{x}, \mathcal{I}_{i,n})]) {\mathcal{J}}_{i,n}^{(l)}]
  &=0,
\end{align*}}
for all $1\leq l\leq d$, where ${\mathcal{J}}_{i,n}^{(l)} ={X}^{(l)}_{\frac{i+m}{n}}-{X}^{(l)}_{\frac{i-1}{n}}$. \\ \\ 
A phase transition occurs as the value of $\alpha$ reaches the critical point $\alpha=2$, instance in which the process $\textbf{X}$   reaches the domain of attraction of a Gaussian distribution. It is common to observe that in this particular case, the process under consideration possesses moments of arbitrary order, which makes it natural for the following technical condition to be required for some special instances:\\

\noindent \textit{Hypothesis }$\bm{H_3}$\\
{\rev When $\alpha =2$, hypothesis $\bm{H_2}(2)$ holds.} In addition, there exists $\gamma>1$ such that 
\begin{align}\label{eq:Xtails2} 
\Pb[|X_{t}^{(1)}| \geq s]
  &\leq\kappa (ts^{-2})^{\gamma}.
\end{align}

\section{Main results} \label{sec1.3}
\setcounter{equation}{0}
\renewcommand{\theequation}{\thesection.\arabic{equation}}

In this section, we describe our main findings, whose presentation we split into sections addressing the functional behavior and quantitative estimates respectively.\\

\noindent\textbf{Asymptotic functional description of $Z^{(n)}$}\\
In the sequel, we will require  the notion of finite dimensional \textit{stable convergence}, which is originally due to Renyi \cite{R63}. Let $\mathcal G=\si(X_t: ~t\geq0)$ where $X$ is a L\'evy process with characteristic exponent as  described in Section \ref{sec1.1}. Let $r\geq 1$ be a given positive integer. For $\R^r$-valued 
random variables $(Y_n)_{n\geq 1}$ and
$Y$ we say that $Y_n$ converges $\mathcal G$-stably in law to $Y$, where $Y$ is defined on an extension  
$\probp$ of the original probability space $\prob$, if and only if for  any continuous bounded function $f:\R^r \to \R$ and any bounded
$\mathcal G$-measurable random variable $F$ it holds that
\bee \label{gstab}
\lim_{n \to \infty} \E[Ff(Y_n)] = \E'[Ff(Y)].
\eee
In this case we write $Y^n \fidi Y$.\\

 The next result not only gives a sufficient Stein type condition for asymptotic mixed normality, but also guarantees stable convergence. We demonstrate it in a rather general setting as it may have interest in its own right. Its proof is given in Section \ref{s: proof stable}. 

\begin{theo} \label{stein}
Let $\mathcal G$ be a sub-$\si$-algebra of $\mathcal F$. We consider $\R^r$-valued random variables $(S_n)_{n\geq 1}$  and a random diagonal matrix 
$\Si=\text{\rm diag}(V_1^2,\ldots,V_r^2)$, which is assumed to be $\mathcal G$-measurable. Furthermore, we assume that $\sup_n \E[\|S_n\|^2]<\infty$,
$\E[V_j^2]<\infty$ for all $j=1,\ldots,r$, and all $V_j^2$'s have no positive mass at $0$.  If the convergence 
\bee \label{condi}
\lim_{n\to \infty} \sum_{k=1}^r \left( \E\left[FS_n^k \frac{\partial}{\partial x_k} h(S_n)\right] - \E\left[F V_k^2 \frac{\partial^2}{\partial x_k^2} h(S_n)\right] \right) \to 0
\eee
holds for all bounded $\mathcal G$-measurable random variables $F$ and all functions $h\in C^2(\R^r;\R)$ with $\|\nabla h\|_{\infty} + \|\nabla^2 h\|_{\infty}<\infty$,
then we obtain the stable convergence
\bee \label{stabcon}
S_n \fidi S=\Si^{1/2} N,
\eee
where $N\sim \mathcal{N}_{r}(0,\text{id})$ is a standard $r$-dimensional normal variable defined on an extended space and independent of $\mathcal{G}$. In particular, $S$ has a mixed normal distribution with mean $0$ and conditional variance $\Si$. 
\end{theo} 

\noindent
We remark that condition \eqref{condi} does not provide the rate of convergence towards the limit $S$ (in contrast to classical Stein's method). This is not surprising as stable convergence is not  metrizable in general. 

\begin{rem} \rm
The paper \cite{NR09} discusses another technique to show stable convergence of weighted functionals of fractional Brownian motion towards a mixed normal distribution, which is somewhat similar in spirit to the statement Theorem \ref{stein}. The authors consider a conditional characteristic function of their statistic and prove via Malliavin calculus that it asymptotically satisfies the differential equation determined by the conditional characteristic function of the limit.
\qed  
\end{rem}

\begin{rem} \rm
Theorem \ref{stein} has multiple applications, as \eqref{condi} provides a convenient condition for obtaining stable convergence in several settings. We briefly demonstrate an example 
of application of Theorem \ref{stein} in the Gaussian setting. The interested reader is referred to \cite{NP12} for a detailed discussion of the basic operators in the theory of Malliavin calculus, to be introduced next.\\

\noindent Consider  a univariate isonormal Gaussian process $X = \left \{ X(h): \, h \in \mathbb{H} \right \}$ associated with a real separable infinite-dimensional Hilbert space $\mathbb{H}$ and a sequence of Skorohod integrals $S_n = \delta(u_n)$ for some symmetric functions $u_n$ in the Sobolev space $\mathbb{D}^{2,2}(\mathbb{H})$.  The statement of  \cite[Theorem 3.1]{NoNu} gives
sufficient conditions that guarantee asymptotic mixed normality for the sequence $S_n$. More specifically, they show that conditions
\begin{align} \label{Mallcond1}
&\lan u_n, h \ran_{\mathbb H}  \xrightarrow{L^1}0 \qquad \text{for all } h \in \mathbb{H}, \\
& \label{Mallcond2}\lan u_n, DS_n \ran_{\mathbb H}  \xrightarrow{L^1} \Sigma,
\end{align}
where $D$ denotes the Malliavin derivative,
imply the stable convergence $S_n \fidi S=\Si^{1/2} N$, where $N\sim \mathcal{N}(0,1)$ is defined on an extended space and is independent of $\mathcal{G}$. This result easily follows
from our Theorem \ref{stein} via checking condition \eqref{condi}. 
Indeed, the integration by parts formula gives the identity
\bee
\E[FS_n h'(S_n)] = \E[F\de(u_n) h'(S_n)] =\E\left[\lan u_n, h'(S_n) DF + Fh''(S_n) DS_n \ran_{\mathbb{H}}\right].
\eee
For $F=g(X(h_1), \ldots, X(h_p))$ with $h_i \in \mathbb{H}$, 
$g\in C^{ 1}(\R^p;\R)$ and $\|g\|_{\infty} + 
\|\nabla g\|_{\infty}<\infty$, condition \eqref{Mallcond1} implies that 
\bee
\left|\E\left[\lan u_n, h'(S_n) DF  \ran_{\mathbb{H}}\right]\right| \leq \|h'\|_{\infty} 
\E\left[|\lan u_n, DF  \ran_{\mathbb{H}}|\right] \to 0.
\eee
By an approximation argument the above convergence holds for any bounded random variable 
$F$. 
On the other hand, via condition  \eqref{Mallcond2}, we conclude that 
\bee
\left| \E\left[Fh''(S_n) (\lan u_n,  DS_n \ran_{\mathbb{H}} - \Sigma)\right] \right|
\leq C \|h''\|_{\infty} \E\left[|\lan u_n,  DS_n \ran_{\mathbb{H}} - \Sigma|\right] \to 0.
\eee
Hence, Theorem \ref{stein} implies the stable convergence $S_n \fidi S=\Si^{1/2} N$. \qed
\end{rem}

Now we present a stable limit theorem for the statistic $Z^{(n)}$.

\begin{theo}\label{th2} 
Let $\mathfrak{g}_n$ and $\mathfrak{g}$ be defined by \eqref{eq:gnvariancedef} and \eqref{eq:gvariancedef}, respectively. Let $\bm{X}$ be a L\'evy process such that $\bm{H}_{1}(\alpha)$ holds for $\alpha\in(0,2)$. Assume that
\begin{equation}{\label{eq: var to zero}}
 \sup_{\bm{x}\in\R^{d}}|\mathfrak{g}_n(\bm{x})-\mathfrak{g}(\bm{x})| \rightarrow 0 
\end{equation}
for $n \rightarrow \infty$. Then, there exists a Brownian motion $W$ independent of $\bm{X}$ such that the process $Z^{(n)}=(Z_{t}^{(n)}\ ;\ t\geq0)$ satisfies 
\begin{align} \label{convfidi}
Z^{(n)} 
  &\fidi \Big(\int_0^t\sqrt{A_{s}}W(ds)\ ;\ t\geq0 \Big),
\end{align}
where $\{A_s\}_{s\geq 0}$ is given as in \eqref{eq:Atdef} and the convergence takes place under the topology of convergence of finite dimensional distributions. 
\end{theo}

\begin{rem} \rm
To guarantee the validity of the condition \eqref{eq: var to zero} it suffices to assume that 
$(X_t^{(1)})_{t\geq 0}$ is locally $\al$-stable, i.e. $a_nX_{1/n}^{(1)} \schw L_1$ where $(L_t)_{t\geq 0}$ is a $\al$-stable L\'evy process. The full characterisation of locally $\al$-stable
L\'evy processes is given in \cite[Theorem 2]{I18}. \qed
\end{rem}

\begin{rem} \rm
When $m=0$ Theorem  \ref{th2} can be proved by martingale methods via \cite{J97} as
it has been highlighted in Remark \ref{Jacodst}.  In this case the approach of \cite{J97} is easier
to apply than our method, and moreover it provides functional stable convergence in \eqref{convfidi}. On the other hand, when $m\geq 1$ the methodology 
in \cite{J97} can not be applied directly (as it possibly requires an additional blocking technique) and in this framework the complexity of both methods are comparable. \qed   
\end{rem}

\noindent\textbf{Quantitative assessments}\\
\noindent
In addition to weak convergence of Theorem \ref{th2}  we now assess the 
distance between $Z_t^{(n)}$ as introduced in \eqref{eq:statfirstapp} and its mixed Gaussian limit. Our idea is based upon a combination of Stein's method and the $K$-function approach.
The main results are as follows.  

\begin{theo}\label{th1} 
Let $\bm{X}$ be a  L\'evy process satisfying $\bm{H}_{1}(\alpha)$. Let $A_{t}$ be given by \eqref{eq:Atdef} and define 
\begin{align}{\label{eq: def beta}}
\beta_{n}
  &:=\sup_{\bm{x}\in\R^{d}}|\mathfrak{g}_n(\bm{x})-\mathfrak{g}(\bm{x})|.
\end{align}
Then for all $t>0$, there exists a standard Brownian motion $W$ independent of $\bm{X}$, defined on an extended probability space, such that the following inequalities hold 
\[{\rev
d \Big( Z_{t}^{(n)}, \int_0^{t}\sqrt{A_{s}}W(ds) \Big)
  \le C r_n,}
\]
where
\begin{enumerate}
\item If $\alpha\in(0,2)$, then  
\bee
r_n= \beta_{n}+n^{\frac{1}{2}-\frac{1}{\alpha}}+n^{-1/2}(1+\Indi{\{\alpha=1\}}\log(n)).
\eee
If $\alpha\in(1,2)$ and  $\bm{H}_{2}(\alpha)$ holds, then
\bee
r_n=n^{-1/2}+\beta_{n}.
\eee

\item If $\alpha=2$, then under the condition $\bm{H}_{2}(2)$,
\bee\label{secondpartrnbound}
r_n=n^{-1/2}{\rev (\log(n))^\frac{3}{2}}+\beta_{n},
\eee
and under $\bm{H}_{3}$, 
\bee
r_n= n^{-1/2}+\beta_{n}.
\eee
\end{enumerate} 
\end{theo}


{
\begin{rem}
In the case where $\alpha=2$, the tail probabilities of the variable $X_1^{(1)}$ decay exponentially fast. This observation, combined with the $1/2$-self similarity of $X$,  guarantees the validity of  \eqref{secondpartrnbound} with $\beta_n=0$, yielding the bound \eqref{simplifiedmainparttwogauss}.
\end{rem}
}

\begin{rem} \rm

Here we comment on potential extensions and limitations of Theorem \ref{th1}. First of all, the independence of the components of $\bm{X}$ is not essential for our proofs; it is mainly considered for simplicity of exposition. The tail conditions \eqref{eq:Xtails} and \eqref{eq:Xtails2} can be easily adapted to the setting of general $d$-dimensional locally $\al$-stable L\'evy processes $\bm{X}$. Also extensions to non-equidistant observations seem to be straightforward.\\ 

\noindent The key technique we apply in the proofs is an advanced {version of the $K$-function approach \cite{CheGoShao}. We believe that similar} techniques may work
for models of the form
\[
d\bm{Y}_t = \si_t d\bm{X}_t,
\]
with $\si_t = w(\bm{X}_t)$ or even $\si_t = w(\bm{Y}_t)$, where $w:\R^d \to \R^{d\times d}$ is a bounded and smooth enough function. In this setting one would employ methods of Malliavin calculus to control the dependence structure. 

On the other hand, extensions to a general martingale setting exhibiting a mixed normal limit seem to be out of reach. 

\qed
\end{rem}

\section{Proof of Theorem \ref{stein}} \label{s: proof stable}
\setcounter{equation}{0}
\renewcommand{\theequation}{\thesection.\arabic{equation}}
We first consider a more restricted setting where the random variables $V_1^2,\ldots, V_r^2$ are bounded away from {zero and infinity}. In other words, we assume there exists a $\de\in (0,1)$ such that 
\bee \label{Vbound}
V_1^2,\ldots, V_r^2 \in [\de,\de^{-1}].
\eee
Recalling the definition of stable convergence and noting that $\sup_n \E[\|S_n\|^2]<\infty$, $ \E[\|S\|^2]<\infty$, it suffices to prove the statement
\bee \label{HLip}
\lim_{n \to \infty} \E[Ff(S_n)] = \E'[Ff(S)]
\eee
for any {\rev bounded} $\mathcal G$-measurable random variable $F$ and any Lipschitz function $f:\R^r \to \R$. We will now use a classical result from Stein's method. 
Let $E_{\si}=\text{diag}(\si_1^2,\ldots,\si_r^2)$ be a deterministic diagonal matrix with $\si_j^2>0$. Then, for any Lipschitz function  $f:\R^r \to \R$, there exists 
a $C^2$-solution $h_{\si}:\R^r \to \R$ to the equation
\bee \label{Steineq}
\sum_{k=1}^r \left(x_k \frac{\partial}{\partial x_k} h_{\si}(x) - \si_k^2\frac{\partial^2}{\partial x_k^2} h_{\si}(x) \right)=f(x) - \E[f(\mathcal N_r(0,E_\si))]
\eee 
such that 
\bee \label{boundcond}
\sup_{x\in \R^r} \|\nabla^2 h_{\si}(x)\|_{\infty} \leq \sqrt{r} \|f\|_{\text{Lip}} \frac{\max_j |\si_j|}{\min_j \si_j^2},
\eee
see e.g. \cite[Proposition 4.3.2]{NP12}. Let $t_i^q:=i/q$, $i\geq 0$, $A_i^q:=[t_{i-1}^q, t_i^q)$ and set $H_{{\bf i},q}:=
\prod_{k=1}^r 1_{A_{i_k}^q}(V_k^2)$ for ${\bf i}=(i_1,\ldots,i_r)$. Note that $H_{{\bf i},q}$ is $\mathcal G$-measurable. We will use the decomposition
\begin{align}
 \E[Ff(S_n)] - \E'[Ff(S)]
   = \sum_{i_1,\ldots,i_r\geq 1} \left(\E\left[FH_{{\bf i},q} f(S_n)\right] - \E'\left[FH_{{\bf i},q} f(S)\right] \right).
\end{align}
We remark that the above sum is finite having at most $(q\de^{-1})^r$ non-zero terms due to \eqref{Vbound}. We further decompose  $\E[Ff(S_n)] - \E'[Ff(S)]=R_{q}^1 +R_{n,q}^2$
with
\begin{align}
R_{q}^1
  &:= \sum_{i_1,\ldots,i_r\geq 1} \left(\E'\left[FH_{{\bf i},q} f(E_{t_{{\bf i}-1}^q}^{1/2}N)\right] - \E'\left[FH_{{\bf i},q}
f(\Si^{1/2} N)\right] \right) \\[1.5 ex]
R_{n,q}^2
  &:=  \sum_{i_1,\ldots,i_r\geq 1} \left(\E\left[FH_{{\bf i},q} f(S_n)\right] - \E'\left[FH_{{\bf i},q} f(E_{t_{{\bf i}-1}^q}^{1/2}N)\right] \right),
\end{align}
where $t_{\bf i}^q=(t_{i_1}^q,\ldots, t_{i_r}^q)$, {\rev $E_{t_{{\bf i}-1}^q}^{1/2}$ is a diagonal matrix with diagonal elements equal to $\sqrt{t_{{\bf i}-1}^q}$} and $N\sim \mathcal{N}_d(0,\text{id})$ is independent of $\mathcal G$. We now show that both terms converge to $0$. Since $F$ is bounded and $f$ is a Lipschitz function we readily deduce that 
\bee \label{rn1}
|R_{q}^1|\leq \frac{C}{q}.
\eee 
Using the independence of $N$ and $\mathcal{G}$, and applying the identity  \eqref{Steineq} for $x=S_n$ and 
$E_{\si}=E_{t_{{\bf i}-1}^q}$, we obtain that 
\begin{align}
R_{n,q}^2
  &= \sum_{i_1,\ldots,i_r \geq 1} \left( \E\left[FH_{{\bf i},q} \left(\sum_{k=1}^r S_n^{k} 
\frac{\partial}{\partial x_k} {\rev h_{t_{{\bf i}-1}^q}}(S_n) \right)\right] \right. \\[1.5 ex]
&- \left.  \E\left[FH_{{\bf i},q} \left(\sum_{k=1}^r t_{i_k-1}^q 
\frac{\partial^2}{\partial x_k^2} {\rev h_{t_{{\bf i}-1}^q}}(S_n) \right)\right] \right).
\end{align}
We further decompose $R_{n,q}^2=R_{n,q}^{2.1} + R_{n,q}^{2.2}$ where 
\begin{align}
R_{n,q}^{2.1}
  &:= \sum_{i_1,\ldots,i_r \geq 1} \left( \E\left[FH_{{\bf i},q} \left(\sum_{k=1}^r S_n^{k} 
\frac{\partial}{\partial x_k} {\rev h_{t_{{\bf i}-1}^q}}(S_n) \right)\right] \right. \\[1.5 ex]
&- \left.  \E\left[FH_{{\bf i},q} \left(\sum_{k=1}^r V_k^2  
\frac{\partial^2}{\partial x_k^2} {\rev h_{t_{{\bf i}-1}^q}}(S_n) \right)\right] \right), \\[1.5 ex]
R_{n,q}^{2.2}
  &:= \sum_{i_1,\ldots,i_r \geq 1} \left( \E\left[FH_{{\bf i},q} \left(\sum_{k=1}^r V_k^2  
\frac{\partial^2}{\partial x_k^2} {\rev h_{t_{{\bf i}-1}^q}}(S_n) \right)\right]  \right. \\[1.5 ex]
&- \left.  \E\left[FH_{{\bf i},q} \left(\sum_{k=1}^r t_{i_k-1}^q  
\frac{\partial^2}{\partial x_k^2} {\rev h_{t_{{\bf i}-1}^q}}(S_n) \right)\right] \right)
\end{align}
We observe that $H_{{\bf i},q}\not = 0$ implies that $|V_k^2 - t_{i_k-1}^q|<q^{-1}$. Hence, using \eqref{boundcond}, we conclude that
\bee \label{rn2.2}
|R_{n,q}^{2.2}|
  \leq \frac{C}{q}
\eee
where the constant $C$ depends only on $r,\de, \|f\|_{\text{Lip}}$. Last, since $H_{{\bf i},q}$ is bounded and $\mathcal G$-measurable, condition 
 \eqref{condi} implies the convergence
 \bee \label{rn2.1}
 R_{n,q}^{2.1} \to 0 \qquad \text{as } n\to \infty
 \eee
 for any fixed $q$. Hence, the statement of Theorem \ref{stein} follows from \eqref{rn1},  \eqref{rn2.2} and  \eqref{rn2.1} if we first choose $q$ large and then $n$ large. 
 
 In the last step we drop the additional condition \eqref{Vbound}. For any random variable $F$ and $\de\in (0,1)$, we define $F_{\de}:=F \prod_{k=1}^r 
 1_{[\de,\de^{-1}]} (V_k^2)$. We obtain the decomposition
 \begin{align}
 \E[Ff(S_n)] - \E'[Ff(S)]
   &= \left(  \E[F_{\de}f(S_n)] - \E'[F_{\de}f(S)] \right) \\[1.5 ex]
 &+ \left(  \E[(F-F_{\de})f(S_n)] - \E'[(F-F_{\de})f(S)] \right)
\end{align}
We know from the previous step
that 
\[
\E[F_{\de}f(S_n)] - \E'[F_{\de}f(S)] \to 0 \qquad \text{as } n\to \infty
\]
for any fixed $\de$. On the other hand, using Cauchy-Schwarz inequality, the second term can be bounded as 
\[
\left| \E[(F-F_{\de})f(S_n)] - \E'[(F-F_{\de})f(S)] \right| \leq C \left(\sum_{k=1}^r \P\left(V_k^2\in [\de, \de^{-1}]^{c}\right) \right)^{1/2}
\]
since $f$ is Lipschitz {\rev and $S_n$ and $S$ have uniformly bounded second moments}. The latter converges to $0$ as $\de\to 0$, since all $V_j^2$'s have no positive mass at $0$. The proof is now complete if we choose $\de$
small and then $n$ large. 
\qed

\section{Proof of Theorem \ref{th1}} \label{s: proof rate}
\setcounter{equation}{0}
\renewcommand{\theequation}{\thesection.\arabic{equation}}

\subsection{Stein's approach and some definitions}
 We start the proof by providing a Stein's characterisation of the problem.
Let $t>0$ be fixed. Consider a test function $\psi\in C^{\rev 5}(\R;\R)$ such that $\|\psi^{(\ell)}\|_{\infty}\leq 1$ for $\ell=0,\ldots, 5$. 
By the independence between $\bm{X}$ and  $W$, the law of the stochastic integral $\int_0^{t}\sqrt{A_{s}}W(ds)$
coincides with that of $\sqrt{V}N$ {\rev for $V = \int_0^t A_s ds$}, where $N\sim \mathcal{N}(0,1)$ is  independent of $\bm{X}$. From here it follows that 
\begin{align*}
d \Big( Z_{t}^{(n)}, \int_0^{t}\sqrt{A_{s}}W(ds) \Big)
  &=d( Z_{t}^{(n)},V^{1/2}N).
\end{align*}
In order to bound the right-hand side, we proceed as follows: let $h_{V}:\R\rightarrow\R$ be the random function obtained as the solution to the Stein equation 
\begin{align*}
x  h_{V}^{\prime}(y)
-V h_{V}^{\prime\prime}(y)
  =\psi(y) - \E[\psi(V^{\frac{1}{2}}N)],
\end{align*}
whose existence is guaranteed by \cite[Proposition 4.3.2]{NP12}. The function $h_{V}$ is determined by $\psi$, but we  omit such dependence in the notation for convenience. By first evaluating the equation above at $y=Z_{t}^{(n)}$ and then taking expectation on both sides of the resulting expression, we observe that the problem is reduced to {\rev showing 
\begin{align}\label{eq:goal1simple}
\sup_{\psi\in \mathcal{H}}\Big|\E\left[Z_{t}^{(n)} h_V^{\prime}(Z_{t}^{(n)})
  - V h_V^{\prime\prime}(Z_{t}^{(n)})\right] \Big| \leq C r_n.
\end{align}
}

\noindent
{\rev In the sequel we will need the following lemma, which determines the behaviour of $h'_{\ga}(x)$ as a function of $\ga$.

\begin{lemma} \label{hanalysis}
For $\psi\in \mathcal{H}$ let $h_{\ga}$, $\ga>0$, denote the unique bounded solution 
to the Stein equation 
\begin{align*}
x h_{\gamma}^{\prime}(x)-\gamma  h_{\gamma}^{\prime\prime}(x)
  &=\psi(x) - \E[\psi(\sqrt{\gamma}N)].
\end{align*}
Define the quantity 
\begin{align} \label{Vdefi}
\mathcal{V}[\psi](x)
  &:=-\int_0^{\infty}e^{-\theta}(1-e^{-2\theta})\E\left[\psi^{(3)}(e^{-\theta}x+\sqrt{1-e^{-2\theta}}{\gamma}N)\right]d\theta.
\end{align} 
Then, for any {$\tilde{\gamma}>0,$} 
\begin{align}\label{eq:TayloresoltoSteineq}
\left|h_{\gamma}^{\prime}(x)- h_{\tilde{\gamma}}^{\prime}(x)-(\gamma-\tilde{\gamma})\mathcal{V}[\psi](x)\right|
  &\leq C\int_0^{\infty}e^{-\theta}\|\psi^{(5)}\|_{\infty}|\gamma-\tilde{\gamma}|^{2}d\theta,
\end{align}
{for some constant $C>0$.}
\end{lemma}

\begin{proof}
See Supplementary Material.
\end{proof}
}

\noindent  In order to prove \eqref{eq:goal1simple}, we will now implement a methodology close in spirit to the $K$-function approach (see \cite[Section 2.3.1]{CheGoShao}), suitably adapted to the case of non-central limit theorems.
Using the independent increments property  of $\bm{X}$, it follows that 
\begin{align*}
Z_{t}^{(n)}
  &=  \frac{1}{\sqrt{n}}\sum_{i=1}^{\lfloor nt\rfloor } f_n(\bm{X}_{\frac{i-1}{n}} ,\mathcal{I}_{i,n}),
\end{align*}
with 
\begin{align*}
f_n(\bm{x} ,\bm{y})
  &:=g ( \bm{x} ,\bm{y}) - \E\left[g\left (\bm{x}, a_n \Tilde{\bm{X}}_{\frac{1}{n}}^{(0)}, ... , a_n \Tilde{\bm{X}}_{\frac{1}{n}}^{(m)} \right)\right],
\end{align*}
where $\bm{\tilde{X}}^{(0)}, ... , \bm{\tilde{X}}^{(m)}$ are independent copies of $\bm{X}$, independent of $W$. Let us introduce for $i=1,\dots, \lfloor n t\rfloor$, the process
$\dot{\bm{X}}[i]=(\dot{\bm{X}}_{t}[i]\ ;\ t\geq0)$ with $\dot{\bm{X}}_{t}[i]=(\dot{{X}}_{t}^{(1)}[i],\dots, \dot{{X}}_{t}^{(d)}[i])$, defined by 
{
\begin{align}\label{eq:Xdotdef}
\dot{{X}}_{t}^{(\ell)}[i]
  &:=\int_0^{t}\Indi{ \R\backslash[\frac{i-1}{n},\frac{i + m}{n}]}(s)dX_{s}^{(\ell)}+\int_0^{t}\Indi{[\frac{i-1}{n},\frac{i + m}{n}]}(s)d\tilde{X}_{s}^{(\ell)}.
\end{align}}
Strictly speaking, $\dot{\bm{X}}[i]$ depends as well on $n$, but for convenience, we have avoided making this dependence explicit in the notation. Clearly, $\dot{\bm{X}}_{t}[i]$ is independent of $\mathcal{I}_{i,n}$. Define
\begin{align*}
\dot{V}[i]
  &:=\int_{0}^{t}\mathfrak{g}(\dot{\bm{X}}_{s}[i])ds
	\\
\dot{\mathcal{I}}_{j,n}[i]
  &:=a_n\left(\dot{\bm{X}}_{\frac{j}{n}}[i]-\dot{\bm{X}}_{\frac{j-1}{n}}[i]\right)\\
\dot{Z}_{t}^{(n)}[i]
  &:=\frac{1}{\sqrt{n}}\sum_{j=1}^{\lfloor nt\rfloor} f_n\left( \dot{\bm{X}}_{\frac{j-1}{n}}[i] ,
\dot{\mathcal{I}}_{j,n}[i]\right),
\end{align*}
for $j=1,\dots, n$. In the next lemma we demonstrate some key inequalities for the errors $V-\dot{V}[i]$ and $Z_t^{(n)}-\dot{Z}_{t}^{(n)}[i]$. {\rev Notice that, by the definition of $\dot{\bm{X}}_{s}[i]$, it is 
\begin{equation}{\label{eq:1new}}
 \bm{X}_{s} - \dot{\bm{X}}_{s}[i] = 
 \begin{cases}
    0 \qquad \mbox{ for } s \le \frac{i-1}{n}, \\
    \bm{X}_{s} - \bm{X}_{\frac{i-1}{n}}{-\bm{\tilde{X}}_{s} + \bm{\tilde{X}}_{\frac{i-1}{n}}} \qquad \mbox{ for } \frac{i-1}{n} \le s \le \frac{i + m}{n} \\
   \bm{X}_{\frac{i + m}{n}} - \bm{X}_{\frac{i-1}{n}}
   {-\bm{\tilde{X}}_{\frac{i + m}{n}} + \bm{\tilde{X}}_{\frac{i-1}{n}}}\qquad \mbox{ for } s \ge \frac{i + m}{n}.
 \end{cases}
\end{equation}
}

\begin{lemma}{\label{lemma: L2 z-zdot}}
If the L\'evy process  $\bm{X}$  satisfies assumption $\bm{H}_1(\alpha)$, then
\begin{align}\label{eq:VminusVdotTaylorone}
\left|V-\dot{V}[i]\right|
  &\leq C{\int_0^t\left(1\wedge \left \|\bm{X}_{s}-\dot{\bm{X}}_{s}[i]\right \|\right)ds},
\end{align}
\begin{align}\label{eq:VminusVdotTaylor}
\left|V-\dot{V}[i]-\sum_{\ell=1}^{d}\int_{\frac{i-1}{n}}^{t}{(X_{s}^{(\ell)}-\dot{X}_{s}^{(\ell)}[i])}\frac{\partial \mathfrak{g}}{\partial x_{\ell}}(\dot{\bm{X}}_{s}[i])ds\right|
  \leq C {\int_0^t\left({\left\|\bm{X}_{s}-\dot{\bm{X}}_{s}[i]\right \|}\wedge \left \|\bm{X}_{s}-\dot{\bm{X}}_{s}[i]\right \|^2\right)ds},
\end{align}
and
\begin{equation}{\label{eq: bound increment Z}}
\E \Big{[} \Big{|}    {Z}_{t}^{(n)} - \dot{Z}_{t}^{(n)}[i]\Big{|}^2 \Big{]}\le C n^{-1}\left(1+\Indi{\{\alpha=2\}}\log(n)\right).
\end{equation}
Moreover, when $\bm{H}_{3}$ holds, the logarithm above can be removed.
\end{lemma}

\begin{proof}
See Supplementary Material.
\end{proof}

\subsection{The main decomposition}
 In this subsection we introduce the main decomposition of the quantity defined in \eqref{eq:goal1simple}.
Observe that  
\begin{align*}
\sum_{i=1}^{\lfloor nt \rfloor}\E\left[f_n( {\bm{X}}_{\frac{i-1}{n}} ,
\mathcal{I}_{i,n})h_{\dot{V}[i]}^{\prime}\left(\dot{Z}_{t}^{(n)}[i]\right)\right]=0.
\end{align*}
This follows from the tower property of conditional expectation, the fact that  $h_{\dot{V}[i]}^{\prime}\left(\dot{Z}_{t}^{(n)}[i]\right)$ is measurable with respect to $\mathcal{F}^{(i)}$ that has been defined in \eqref{sialg} and can be seen as $\sigma( \bm{\dot{X}}_{t}[i]\ ;\ t\geq 0)$, and the identity 
\begin{equation}{\label{eq: 5.8.5}}
\E[f_n( {\bm{X}}_{\frac{i-1}{n}} ,\mathcal{I}_{i,n}) |\mathcal{F}^{(i)} ]= 0.   
\end{equation}
From here we get
\begin{align*}
\E\left[Z_{t}^{(n)} h_{V}^{\prime}(Z_{t}^{(n)})\right]
  &=\frac{1}{\sqrt{n}}
\sum_{i=1}^{\lfloor nt\rfloor}\E\left[f_n( {\bm{X}}_{\frac{i-1}{n}} ,
\mathcal{I}_{i,n})
\left(h_{V}^{\prime}(Z_{t}^{(n)})
- h_{\dot{V}[i]}^{\prime}(\dot{Z}_{t}^{(n)}[i])\right)\right].
\end{align*}
Consequently,  we obtain the decomposition
\begin{align}\label{eq:SnIBPintro}
&\E\left[Z_{t}^{(n)} h_{V}^{\prime}(Z_{t}^{(n)})\right] \nonumber \\
&=\frac{1}{\sqrt{n}}
\sum_{i=1}^{\lfloor nt\rfloor}\E\left[f_n( {\bm{X}}_{\frac{i-1}{n}} ,
\mathcal{I}_{i,n})
\left(h_{V}^{\prime}(Z_{t}^{(n)})
- h_{V}^{\prime}(\dot{Z}_{t}^{(n)}[i])\right)\right] +R_{1,n},
\end{align}
where 
\begin{align}\label{eq:R1ndef}
R_{1,n}
  &:=\frac{1}{\sqrt{n}}
\sum_{i=1}^{\lfloor nt\rfloor}\E\left[f_n( {\bm{X}}_{\frac{i-1}{n}} ,
\mathcal{I}_{i,n})
\left(h_{V}^{\prime}(\dot{Z}_{t}^{(n)}[i])- h_{\dot{V}[i]}^{\prime}(\dot{Z}_{t}^{(n)}[i])\right)\right].
\end{align}
{\rev The error term $R_{1,n}$ warrants extra attention, as it plays a crucial role in solving the Stein equation. In Section \ref{sec5.3} we will show that, under conditions of Theorem \ref{th1}, it holds that
\begin{align} \label{r1nst}
|R_{1,n}|\leq Cr_n,
\end{align}
where the constant $C$ does not depend on $h$. 
}
 {Application {\revch of} a Taylor approximation} implies the identity
\begin{align}{\label{eq: Taylor}}
h_{V}^{\prime}(Z_{t}^{(n)})
- h_{V}^{\prime}(\dot{Z}_{t}^{(n)}[i])
&= h''_V \left(\dot{Z}_{t}^{(n)}[i]\right) \left(Z_{t}^{(n)}- \dot{Z}_{t}^{(n)}[i]\right)\\
  & + h_V''' \left(\dot{Z}_{t}^{(n)}[i]+ \tau (Z_{t}^{(n)}- \dot{Z}_{t}^{(n)}[i])\right ) \left(Z_{t}^{(n)}- \dot{Z}_{t}^{(n)}[i]\right)^2 \nonumber
\end{align}
for some $\tau \in [0,1]$. We observe that 
\begin{equation}{\label{eq: Sn 1.25}}
Z_{t}^{(n)}- \dot{Z}_{t}^{(n)}[i]
  = \frac{1}{\sqrt{n}} \sum_{j=1}^{\lfloor nt\rfloor } \left(f_n ( {\bm{X}}_{\frac{j-1}{n}} ,
\mathcal{I}_{j,n}) - f_n ( \dot{\bm{X}}_{\frac{j-1}{n}}[i] ,
\dot{\mathcal{I}}_{j,n}[i])  \right).
\end{equation}
Because of the definition of $\dot{\bm{X}}_{\frac{j-1}{n}}[i]$ and $\dot{\mathcal{I}}_{j,n}[i]$, for $j + m \le i-1$, we have that  $\dot{\bm{X}}_{\frac{j-1}{n}}[i] = {\bm{X}}_{\frac{j-1}{n}}$ and $\dot{\mathcal{I}}_{j,n}[i] = {\mathcal{I}}_{j,n}$. Hence, we can write 
\begin{align}{\label{eq: 5.13.5}}
Z_{t}^{(n)}- \dot{Z}_{t}^{(n)}[i]
  &= \frac{1}{\sqrt{n}} \sum_{j=1}^{\lfloor nt\rfloor } \Indi{\{j \ge i +m+1 \}} \left(f_n( {\bm{X}}_{\frac{j-1}{n}} ,
\mathcal{I}_{j,n}) - f_n ( \dot{\bm{X}}_{\frac{j-1}{n}}[i] ,
\dot{\mathcal{I}}_{j,n}[i])  \right) \\
& + \frac{1}{\sqrt{n}} \sum_{j = i - m}^{i + m} \left(f_n( {\bm{X}}_{\frac{j-1}{n}} ,
\mathcal{I}_{j,n}) - f_n ( \dot{\bm{X}}_{\frac{j-1}{n}}[i] ,
\dot{\mathcal{I}}_{j,n}[i])  \right). \nonumber
\end{align}
Plugging in the above term into \eqref{eq: Sn 1.25}  we obtain, by \eqref{eq:SnIBPintro} and \eqref{eq: Taylor},
\begin{align*}
 \E\left[Z_{t}^{(n)} h'_V (Z_{t}^{(n)})\right] & = \frac{1}{\sqrt{n}} \sum_{i=1}^{\lfloor nt\rfloor} \E \left[ f_n( {\bm{X}}_{\frac{i-1}{n}} ,
\mathcal{I}_{i,n})h''_V(\dot{Z}_{t}^{(n)}[i]) \right.\\
& \times \left( \frac{1}{\sqrt{n}} \sum_{j=1}^{\lfloor nt\rfloor} \Indi{\{j \ge i + m + 1\}} \left(f_n( {\bm{X}}_{\frac{j-1}{n}} ,
\mathcal{I}_{j,n}) - f_n ( \dot{\bm{X}}_{\frac{j-1}{n}}[i] ,
\dot{\mathcal{I}}_{j,n}[i])  \right) \right. \\
& \left.+ \frac{1}{\sqrt{n}} \sum_{j = i - m}^{i + m} \left(f_n( {\bm{X}}_{\frac{j-1}{n}} ,
\mathcal{I}_{j,n}) - f_n ( \dot{\bm{X}}_{\frac{j-1}{n}}[i] ,
\dot{\mathcal{I}}_{j,n}[i])  \right) \right)  \\
& \left. + {\rev h'''_V \left(\dot{Z}_{t}^{(n)}[i] + \tau (Z_{t}^{(n)} - \dot{Z}_{t}^{(n)}[i])\right) \left(Z_{t}^{(n)} - \dot{Z}_{t}^{(n)}[i]\right)^2 } \right]
 + R_{1,n} \\
& = {\rev T_{1,n} + T_{2,n} + T_{3,n} + T_{4,n} + R_{1,n}},
\end{align*}
{\rev where $R_{1,n}$ is as in \eqref{eq:R1ndef} and $T_{1,n},T_{2,n},T_{3,n}$ and $T_{4,n}$ are given by}
\begin{align*}
{\rev T_{1,n}}
  &{\rev :=\frac{1}{n} \sum_{i=1}^{\lfloor nt\rfloor}\sum_{j=1}^{\lfloor nt\rfloor} \Indi{\{j \ge i + m + 1\}}\E \left[ f_n( {\bm{X}}_{\frac{i-1}{n}} ,
\mathcal{I}_{i,n})h''_V(\dot{Z}_{t}^{(n)}[i]) \left(f_n( {\bm{X}}_{\frac{j-1}{n}} ,
\mathcal{I}_{j,n}) - f_n ( \dot{\bm{X}}_{\frac{j-1}{n}}[i] ,
\dot{\mathcal{I}}_{j,n}[i])  \right)  \right] }\\
{\rev T_{2,n} }
  &{\rev :=\frac{1}{n} \sum_{i=1}^{\lfloor nt\rfloor}\sum_{j = i - m}^{i + m} \E \left[ f_n( {\bm{X}}_{\frac{i-1}{n}} ,
\mathcal{I}_{i,n})f_n( {\bm{X}}_{\frac{j-1}{n}} ,
\mathcal{I}_{j,n}) h''_V({Z}_{t}^{(n)})\right] }\\
{ \rev T_{3,n}}
  &{\rev := \frac{1}{\sqrt{n}} \sum_{i=1}^{\lfloor nt\rfloor} \E \left[ h'''_V \left(\dot{Z}_{t}^{(n)}[i] + \tau (Z_{t}^{(n)} - \dot{Z}_{t}^{(n)}[i])\right)\left(Z_{t}^{(n)} - \dot{Z}_{t}^{(n)}[i]\right)^2 \right]} \\
 {\rev  T_{4,n}}
  & {\rev :=\frac{1}{n} \sum_{i=1}^{\lfloor nt\rfloor}\sum_{j = i - m}^{i + m} \E \left[ f_n( {\bm{X}}_{\frac{i-1}{n}} ,
\mathcal{I}_{i,n})(h''_V(\dot{Z}_{t}^{(n)}[i])- h''_V({Z}_{t}^{(n)}) )f_n( {\bm{X}}_{\frac{j-1}{n}} ,
\mathcal{I}_{j,n}) \right].}
\end{align*}
{\rev We will prove in Section \ref{sec5.4} that the term $T_{2,n}$ provides the main contribution. In particular, we will show that
\begin{align} \label{termT2}
\sup_{\psi\in \mathcal{H}} \left|T_{2,n}- \E\left[ V h_V^{\prime\prime}(Z_{t}^{(n)})\right]\right|\leq Cr_n.
\end{align}
On the other hand, we will show in the following proposition that the terms $T_{1,n}$, $T_{3,n}$ and $T_{4,n}$ are negligible.

}

{\rev \begin{prop}{\label{prop: T negl}}
 Let $\bm{X}$ be a  L\'evy process satisfying $\bm{H}_{1}(\alpha)$ and let $T_{1,n}$, $T_{3,n}$ and $T_{4,n}$ be as above.
Then, there exists a constant $C>0$, not depending on $h$ or $n$, such that the following inequalities hold 
\begin{enumerate}
\item If $\alpha\in(0,2)$, then  
$$|T_{1,n} + T_{3,n} + T_{4,n}| \le C(n^{\frac{1}{2}-\frac{1}{\alpha}} + n^{-1/2} (1+\Indi{\{\alpha=1\}}\log(n)).$$

\item Under the assumption that $\alpha\in(1,2]$ and $\bm{H}_{2}(\alpha)$ is satisfied, {the following inequality is satisfied}
$$|T_{1,n} + T_{3,n} + T_{4,n}| \le Cn^{-1/2}(1+\Indi{\{\alpha=2\}} (\log(n))^{\frac{3}{2}}).$$
Moreover, the logarithm can be removed if $\bm{H}_3$ holds.
\end{enumerate}
\end{prop}
\begin{proof}
See Supplementary Material. 
\end{proof}

The combination of Proposition \ref{prop: T negl} with \eqref{r1nst} and \eqref{termT2} completes the proof of Theorem
\ref{th1}.

}

\subsection{{\rev Treatment of the term $R_{1,n}$}} \label{sec5.3}
Combining \eqref{eq:R1ndef} and \eqref{eq:TayloresoltoSteineq}, we obtain 
\begin{align*}
|R_{1,n}|
  &\leq R_{1,1,n}+R_{1,2,n},
\end{align*}
where 
\begin{align}
R_{1,1,n}
  &:=\Big|\frac{C}{\sqrt{n}}\sum_{i=1}^{\lfloor nt \rfloor } \E\left[{|f_n( {\bm{X}}_{\frac{i-1}{n}} ,
\mathcal{I}_{i,n})|}|V-\dot{V}[i]|^{2}\right]\Big|\label{eq:E1n1def}\\
R_{1,2,n}
  &:=\Big|\frac{C}{\sqrt{n}}\sum_{i=1}^{\lfloor nt \rfloor }\E\left[f_n( {\bm{X}}_{\frac{i-1}{n}} ,
\mathcal{I}_{i,n})(V-\dot{V}[i])\mathcal{V}[\psi](\dot{Z}_{t}^{(n)}[i])\right] \Big| \label{eq:E1n2def}.
\end{align}
Using \eqref{eq:E1n2def} and \eqref{eq:VminusVdotTaylor}, we get 

\begin{align}{\label{eq: 5.14.5 R12n}}
R_{1,2,n}
  &{\rev \leq \left|\frac{C}{\sqrt{n}}\sum_{i=1}^{\lfloor nt \rfloor }\sum_{\ell=1}^{d}\int_{\frac{i-1}{n}}^{t}
	\E\left[f_n ( {\bm{X}}_{\frac{i-1}{n}},\mathcal{I}_{i,n}) (X_{s}^{(\ell)}-\dot{X}_{s}^{(\ell)}[i])\frac{\partial \mathfrak{g}}{\partial x_{\ell}}(\dot{\bm{X}}_{s}[i])\mathcal{V}[\psi](\dot{Z}_{t}^{(n)}[i])\right]ds \right|	} \nonumber \\	
	&+ \frac{C}{\sqrt{n}} \sum_{i=1}^{\lfloor nt \rfloor }\E\left[{\int_0^t\left(1\wedge \left \|\bm{X}_{s}-\dot{\bm{X}}_{s}[i]\right \|^2\right)ds} \left|\mathcal{V}[\psi](\dot{Z}_{t}^{(n)}[i]) \right|  \right], 
\end{align}
while \eqref{eq:E1n2def} and \eqref{eq:VminusVdotTaylorone} provide
\begin{align}{\label{eq: 5.14.75 R12n}}
R_{1,2,n}
  &\leq 	\frac{C}{\sqrt{n}}\sum_{i=1}^{\lfloor nt \rfloor }\E\left[{\int_0^t\left(1\wedge \left \|\bm{X}_{s}-\dot{\bm{X}}_{s}[i]\right \|\right)ds} \Big|\mathcal{V}[\psi](\dot{Z}_{t}^{(n)}[i]) \Big|  \right] \\
  & \leq \frac{C}{\sqrt{n}}\sum_{i=1}^{\lfloor nt \rfloor }\E\left[{\int_0^t\left(1\wedge \left \|\bm{X}_{s}-\dot{\bm{X}}_{s}[i]\right \|\right)ds} \right], \nonumber 
\end{align}
{\rev where we have used that $\|\mathcal{V}[\psi]\|_{\infty}\leq \|\psi^{\prime\prime\prime}\|_{\infty}$.}

\noindent
We now split the argument into two regimes according to whether $g$ satisfies condition $\bm{H}_2(\alpha)$ or not. 

\subsubsection{Case $\alpha\in(0,2)$ under condition $\bf{H}_1(\alpha)$}
 Assume that $\alpha\in(0,2)$ and hypothesis $\bm{H}_{1}(\alpha)$ is satisfied. Here we use the bound gathered in \eqref{eq: 5.14.75 R12n}. 
A combination of hypothesis $\bm{H}_{1}(\alpha)$, the definition of $\bm{X}_{s}-\dot{\bm{X}}_{s}[i]$ as in \eqref{eq:1new} and \eqref{eq:exponewedgexv2} in Lemma \ref{eq:ineqoneminx}, gives the estimate
\begin{align*}
\E\left[ {\int_0^t\left(1\wedge \left \|\bm{X}_{s}-\dot{\bm{X}}_{s}[i]\right \|\right)ds}  \right]
  &\leq C{(n^{-\frac{1}{\alpha}}+n^{-1}\left(1+\Indi{\{\alpha=1\}}\log(n)\right))},
\end{align*}
yielding 
\begin{align}\label{eq:R12n4t23}
R_{1,2,n}
  &\leq C(n^{\frac{1}{2}-\frac{1}{\alpha}}+n^{-1/2}\left(1+\Indi{\{\alpha=1\}}\log(n)\right)).
\end{align}
The term $R_{1,1,n}$ can be handled in a similar fashion, combining \eqref{eq:VminusVdotTaylorone} and the boundedness of $f_{n}$, giving the inequality 
\begin{align}\label{eq:R12n4t24}
R_{1,1,n}
  &\leq C(n^{\frac{1}{2}-\frac{1}{\alpha}}+n^{-1/2}\left(1+\Indi{\{\alpha=1\}}\log(n)\right)).
\end{align}
From here it follows that we have the estimate
\begin{align}\label{eq:R1nfinalboundmay}
R_{1,n}
  &\leq 	C\left(n^{\frac{1}{2}-\frac{1}{\alpha}}+ n^{-1/2}(1+\Indi{\{\alpha=1\}}\log(n))\right),
\end{align}
{\rev and hence $R_{1,n}\leq Cr_n$.}

\subsubsection{Case $\alpha\in(1,2]$ under condition $\bf{H}_2(\alpha)$}
\noindent {\rev Assume that $\alpha\in(1,2]$ and hypothesis {$\bm{H}_2(\alpha)$} is satisfied. We consider \eqref{eq: 5.14.5 R12n} and split the integral between $\frac{i-1}{n}$ and $t$ into two parts. This division is based on the different behavior exhibited by $X_{s}^{(\ell)}-\dot{X}_{s}^{(\ell)}[i]$ depending on whether $s$ is greater than $\frac{i+ m}{n}$ or not. Consequently, we can translate \eqref{eq: 5.14.5 R12n} to the following: }
\begin{align*}
R_{1,2,n}
  &\leq \Big|\frac{{\rev C}}{\sqrt{n}}\sum_{i=1}^{\lfloor nt \rfloor }\sum_{\ell=1}^{d}\int_{\frac{i+m}{n}}^{t}
	\E\left[f_n( {\bm{X}}_{\frac{i-1}{n}},\mathcal{I}_{i,n}){ (X_{s}^{(\ell)}-\dot{X}_{s}^{(\ell)}[i])}\frac{\partial\mathfrak{g}}{\partial x_{\ell}}(\dot{\bm{X}}_{s}[i])\mathcal{V}[\psi](\dot{Z}_{t}^{(n)}[i])\right]ds \Big|	\\
   &+ \Big|\frac{{\rev C}}{\sqrt{n}}\sum_{i=1}^{\lfloor nt \rfloor }\sum_{\ell=1}^{d}\int_{\frac{i-1}{n}}^{\frac{i+m}{n}}
	\E\left[f_n( {\bm{X}}_{\frac{i-1}{n}},\mathcal{I}_{i,n}){ (X_{s}^{(\ell)}-\dot{X}_{s}^{(\ell)}[i])}\frac{\partial\mathfrak{g}}{\partial x_{\ell}}(\dot{\bm{X}}_{s}[i])\mathcal{V}[\psi](\dot{Z}_{t}^{(n)}[i])\right]ds \Big|	\\
	&+ \frac{C}{\sqrt{n}}\sum_{i=1}^{\lfloor nt \rfloor }\E\left[{\int_0^t\left(1\wedge \left \|\bm{X}_{s}-\dot{\bm{X}}_{s}[i]\right \|^2\right)ds}\right],  
\end{align*}
where we used again $\left \| \mathcal{V}[\psi] \right \|_\infty < \infty$.
Recall the definition of $\mathcal{J}_{i,n}=(\mathcal{J}_{i,n}^{(1)},\dots,\mathcal{J}_{i,n}^{(d)})$, given by ${\mathcal{J}}_{i,n}^{(l)} ={X}^{(l)}_{\frac{i+m}{n}}-{X}^{(l)}_{\frac{i-1}{n}}$. Using \eqref{eq:1new} and the uniform boundedness of $f_n$ and $\mathfrak{g}'$, we deduce the bound

\begin{align}{\label{eq: star}}
R_{1,2,n}
  &\leq \Big|\frac{{\rev C}}{\sqrt{n}}\sum_{i=1}^{\lfloor nt \rfloor }\sum_{\ell=1}^{d}\int_{\frac{i{+m}}{n}}^{t}
	\E\left[f_n( {\bm{X}}_{\frac{i-1}{n}},\mathcal{I}_{i,n}){\mathcal{J}_{i,n}^{(\ell)}}\frac{\partial\mathfrak{g}}{\partial x_{\ell}}(\dot{\bm{X}}_{s}[i])\mathcal{V}[\psi](\dot{Z}_{t}^{(n)}[i])\right]ds \Big|	\\	
	&{+\frac{C}{\sqrt{n}}+} \frac{C}{\sqrt{n}}\sum_{i=1}^{\lfloor nt \rfloor }\E\left[{\int_0^t\left(1\wedge \left \|\bm{X}_{s}-\dot{\bm{X}}_{s}[i]\right \|^2\right)ds}\right]  \nonumber,
\end{align}
{\rev where in order to get the second bound on the right hand side we have used that, for $\alpha > 1$, $\E[\left \|\bm{X}_{s}-\dot{\bm{X}}_{s}[i]\right \|]$ is bounded. }

Then, conditioning over $\mathcal{F}^{(i)}$ as defined at \eqref{sialg}, and using that $ \E[f_n({\bm{X}}_{\frac{i-1}{n}},\mathcal{I}_{i,n})\mathcal{J}_{i,n}^{(\ell)}|\mathcal{F}^{(i)}]= \E[f_n({\bm{X}}_{\frac{i-1}{n}},\mathcal{I}_{i,n})\mathcal{J}_{i,n}^{(\ell)}|\mathcal{F}_{\frac{i-1}{n}}] = 0$, 
\begin{align}{\label{eq: H2}}
\E\left[f_n({\bm{X}}_{\frac{i-1}{n}},\mathcal{I}_{i,n}){\mathcal{J}}_{i,n}^{(\ell)}\frac{\partial\mathfrak{g}}{\partial x_{\ell}}\left(\dot{\bm{X}}_{s}[i]\right)\mathcal{V}[\psi]\left(\dot{Z}_{t}^{(n)}[i]\right)\right]=0.
\end{align}
Hence, by \eqref{eq: star} 
\begin{align}\label{eqR12nboundwedge}
R_{1,2,n}
  &\leq {\rev \frac{C}{\sqrt{n}} }\left( 1+	\sum_{i=1}^{\lfloor nt \rfloor }\E\left[{\int_0^t\left(1\wedge \left \|\bm{X}_{s}-\dot{\bm{X}}_{s}[i]\right \|^2\right)ds}\right]\right).
\end{align}
Condition $\bm{H}_{1}(\alpha)$  combined with the fact that $\bm{X}_{s}-\dot{\bm{X}}_{s}[i]$ is as in \eqref{eq:1new} allows us to apply Lemma \ref{eq:ineqoneminx}. It implies the required bound on $R_{1,2,n}$: 
$$ {\rev R_{1,2,n} \le \frac{C}{\sqrt{n}}(1 + \Indi{ \{ \alpha = 2 \} } \log(n))}$$
{\rev and the logarithm can be removed if $\textbf{H}_3$ holds true.} {\rev Indeed, upon satisfying condition $\bm{H}_{3}$, it becomes possible to utilize the inequality derived in \eqref{eq:exponewedgex2} {together with the relation $1\wedge |x|^2\leq |x|\wedge |x|^2$,} as presented in Lemma \ref{eq:ineqoneminx}, in lieu of the one found in \eqref{eq:exponewedgex1wemma}. This substitution yields the desired outcome without necessitating the logarithm. }

For handling $R_{1,1,n}$ as in \eqref{eq:E1n1def}, we use \eqref{eq:VminusVdotTaylorone} to get
\begin{align}\label{eq:VminusVdotTaylor2}
\left|V-\dot{V}[i]\right|^2
  &\leq C {\int_0^t\left(1\wedge \left \|\bm{X}_{s}-\dot{\bm{X}}_{s}[i]\right \|\right)ds}. 
\end{align}
Consequently, by \eqref{eq:E1n1def} and the arguments above we conclude that $R_{1,1,n},R_{1,2,n}$ (and so $R_{1,n}$) are bounded by $c n^{-1/2}(1 + \Indi{ \{ \alpha = 2 \} } \log(n))$ under the hypothesis $\bm{H}_{2}(\alpha)$ and without the logarithm if $\bm{H}_{3}$ holds true. \\
\\
{\rev As a result, $\sup_{\psi \in \mathcal{H}}|R_{1,n}|\leq Cr_n$ in all cases.}

\subsection{Analysis of $T_{2,n}$} \label{sec5.4}
We are left to study the term {$T_{2.n}$}, which is the {\rev dominating} one. We apply the change of variable $\Tilde{j}:= j - i$. Then, 
\begin{align*}
& \frac{1}{n} \sum_{i=1}^{\lfloor n t \rfloor} f_{n} ( {\bm{X}}_{\frac{i-1}{n}} ,\mathcal{I}_{i,n}) \sum_{j = i - m}^{i + m} f_n ( {\bm{X}}_{\frac{j-1}{n}} ,
\mathcal{I}_{j,n}) \\
& = \sum_{\tilde{j} =- m}^{ m} \sum_{i=1}^{\lfloor n t \rfloor} \int_{\frac{i-1}{n}}^{\frac{i}{n}} f_{n} ( {\bm{X}}_{\frac{i-1}{n}} ,\mathcal{I}_{i,n})  f_n ( {\bm{X}}_{\frac{i-1 + \tilde{j}}{n}} , \mathcal{I}_{i + \tilde{j},n}) ds.
\end{align*}
 In order to determine the behavior of $T_{2,n}$, it suffices to bound the Riemann approximation
\begin{align*}
R^{2, n}& = \E \Big[h''_V (Z_{t}^{(n)}) \sum_{\tilde{j} =- m}^{ m} \sum_{i=1}^{\lfloor n t \rfloor} \int_{\frac{i-1}{n}}^{\frac{i}{n}} \Big{(} f_{n} ( {\bm{X}}_{\frac{i-1}{n}} ,\mathcal{I}_{i,n})  f_n ( {\bm{X}}_{\frac{i-1 + \tilde{j}}{n}} , \mathcal{I}_{i + \tilde{j},n}) \\
& -f_n (\bm{X}_s, a_n \Tilde{\bm{X}}_{\frac{1}{n}}^{(1)}, ... , a_n \Tilde{\bm{X}}_{\frac{1}{n}}^{(m)} )f_n (\bm{X}_{s + \frac{\Tilde{j}}{n}}, a_n \Tilde{\bm{X}}_{\frac{1}{n}}^{(1 + \Tilde{j})}, ... , a_n \Tilde{\bm{X}}_{\frac{1}{n}}^{(m + \Tilde{j})} ) \Big{)} ds \Big].
\end{align*}
Adding and removing $f_n (\bm{X}_s, a_n \Tilde{\bm{X}}_{\frac{1}{n}}^{(1)}, ... , a_n \Tilde{\bm{X}}_{\frac{1}{n}}^{(m)} )f_n ( {\bm{X}}_{\frac{i-1 + \tilde{j}}{n}} , \mathcal{I}_{i + \tilde{j},n})$ and using the mean value theorem and the boundedness of $\left \| h''_V \right \|_\infty$, $\left \| \partial_x f \right \|_\infty$, we obtain the following bound: 
\begin{align*}
|R^{2,n}| & \le C \sum_{i=1}^{\lfloor n t \rfloor} \int_{\frac{i-1}{n}}^{\frac{i}{n}} \E[ 1\wedge\|\bm{X}_{s - \frac{i-1}{n}}\|] ds \\
& \le C t {\rev n} \int_0^{1/n}\E[ 1\wedge\|\bm{X}_{s}\|] ds \\
& {\rev \le C t n \int_0^{1/n}(n^{-\frac{1}{\alpha}} + n^{-1}(1 + \Indi{\{\alpha = 1\}} \log(n))) ds } \\
& {\rev \le C \left(n^{- \frac{1}{\alpha}} +  n^{-1}(1 + \Indi{\{\alpha = 1\}} \log(n)) \right)},
\end{align*}
{\rev where we have employed Equation \eqref{eq:exponewedgexv2} in Lemma \ref{eq:ineqoneminx}.} {\rev It 
clearly gives $\sup_{\psi \in \mathcal{H}}|R^{2,n}| \le C r_n$.} 
Finally, we observe that
\begin{align*}
R^{2, 2,n}
  &:= \Big|\E \Big[h''_V (Z_{t}^{(n)}) \int_{0}^{t} \Big{(} 
\sum_{\tilde{j} =- m}^{ m} f_n \Big(\bm{X}_s, a_n \Tilde{\bm{X}}_{\frac{1}{n}}^{(1)}, ... , a_n \Tilde{\bm{X}}_{\frac{1}{n}}^{(m)} \Big)\\
&\times f_n \Big(\bm{X}_{s + \frac{\Tilde{j}}{n}}, a_n \Tilde{\bm{X}}_{\frac{1}{n}}^{(1 + \Tilde{j})}, ... , a_n \Tilde{\bm{X}}_{\frac{1}{n}}^{(m + \Tilde{j})} \Big) ds
-\int_{0}^{t} \mathfrak{g}(\bm{X}_s) ds \Big{)}\Big] \Big|
\end{align*}
satisfies 
\begin{align*}
R^{2, 2,n}
  & \leq  C\sup_{\bm{x}\in\R^{d}}|\mathfrak{g}_n(\bm{x})-\mathfrak{g}(\bm{x})|.
\end{align*}

\noindent {\rev Putting things together we conclude that 
$$\sup_{\psi \in \mathcal{H}} \Big|\E\left[T_{2,n}
  - V h''_V (Z_{t}^{(n)}) \right]\Big| \leq Cr_n$$
in all cases.} The proof of \eqref{termT2} is now complete.
\qed


\section{Supplementary material}
 {\rev This supplement is dedicated to the proof of Proposition \ref{prop: T negl}, Lemmas \ref{hanalysis} and \ref{lemma: L2 z-zdot} previously mentioned, the complexity of which led us to include their proofs in the appendix.  Additionally, it includes the proof of Theorem \ref{th2} and the presentation along with proofs of certain technical lemmas that have been frequently referenced in the main body of the manuscript.}

\section{Proof of Proposition \ref{prop: T negl}}\label{s: proof T negl}
\setcounter{equation}{0}
\renewcommand{\theequation}{\thesection.\arabic{equation}}
{\rev In this section, our objective is to establish Proposition \ref{prop: T negl}, which entails demonstrating the negligible behavior of the three terms $T_{1,n}$, $T_{3, n}$, and $T_{4, n}$. Specifically, we aim to show that these terms tend towards zero at the stated rate. As outlined in the statements of Theorem \ref{th1} and Proposition \ref{prop: T negl}, the convergence rates vary depending on the set of assumptions under consideration. Consequently, we will divide the proof of the proposition into two distinct parts. The initial subsection will focus on the proof under condition $\bm{H_1}(\alpha)$ {when $\alpha\neq 2$}. After that, in the subsequent subsection, we will delve into 
the case where $\alpha \in (1, 2]$ and the stronger condition $\bm{H}_{2}(\alpha)$ is satisfied. }

\subsection{Proof under condition $\bf{H_1}(\alpha)$}
{\rev In this section, we will individually analyze the three terms $T_{1,n}$, $T_{3,n}$, and $T_{4, n}$. We commence by focusing on $T_{1, n}$, which will prove to be the most intricate term to handle.}

\subsubsection{Analysis of $T_{1,n}$}
{Consider the decomposition}
\begin{align*}
{T_{1,n}}
  & =  {T_{1, 2,n}-T_{1,1,n}}.
\end{align*}
where 
\begin{align*}
{T_{1,1,n}}
  &:=\frac{1}{n} \sum_{i=1}^{\lfloor nt\rfloor}\sum_{j=1}^{\lfloor nt\rfloor} \Indi{\{j \ge i +m + 1\}} \E\Big{[} f_n( {\bm{X}}_{\frac{i-1}{n}},
\mathcal{I}_{i,n}) \Big{(}   h''_V (Z_t^{(n)}) - h''_V (\dot{Z}_{t}^{(n)}[i])\Big{)}\\
& \times \Big{(} f_n( {\bm{X}}_{\frac{j-1}{n}} ,
\mathcal{I}_{j,n}) - f_n ( \dot{\bm{X}}_{\frac{j-1}{n}}[i] ,
\dot{\mathcal{I}}_{j,n}[i])   \Big{)} \Big{]}
\end{align*}
and 
\begin{align*}
T_{1,2,n}
&:=\frac{1}{n} \sum_{i=1}^{\lfloor nt\rfloor}\sum_{j=1}^{\lfloor nt \rfloor} \Indi{\{j \ge i + m + 1\}}  \E \Big{[} f_n( {\bm{X}}_{\frac{i-1}{n}} ,
\mathcal{I}_{i,n})   h''_V ({Z}_{t}^{(n)})\\
& \times \Big{(} (f_n ( {\bm{X}}_{\frac{j-1}{n}} ,
\mathcal{I}_{j,n}) - f_n ( \dot{\bm{X}}_{\frac{j-1}{n}}[i] ,
\dot{\mathcal{I}}_{j,n}[i])  ) \Big{)} \Big{]}.
\end{align*}
{\rev Our motivation for splitting the term $T_{1,n}$ as described above lies in the observation that this approach allows us to manage $h''_V ({Z}_{t}^{(n)})$, which remains independent of $i$, as opposed to dealing with $h''_V (\dot{Z}_{t}^{(n)}[i])$. This division proves convenient since it permits us to extract $h''_V ({Z}_{t}^{(n)})$ from the summation, simplifying the process as we can bound it with a constant. Indeed,} because of the boundedness of the second derivatives of $h_V$, we have
\begin{align}{\label{eq: beginning T1}}
{|T_{1,2,n}|}
   \le C\E \Big[ \Big| \frac{1}{n} \sum_{i=1}^{\lfloor nt\rfloor }\sum_{j=1}^{\lfloor nt \rfloor } \Indi{\{j \ge {i+m+1}\}}f_n( {\bm{X}}_{\frac{i-1}{n}},\mathcal{I}_{i,n})  \left(f_n( {\bm{X}}_{\frac{j-1}{n}} ,
\mathcal{I}_{j,n}) - f_n( \dot{\bm{X}}_{\frac{j-1}{n}}[i] ,
\dot{\mathcal{I}}_{j,n}[i])  \right)\Big| \Big].
\end{align}
In order to deal with the expectation above we apply the Lyapunov inequality $\|\cdot\|_{1} \le \|\cdot\|_{2}$, so it suffices to bound
\begin{align}{\label{eq: square T1}}
 {\rev |T_{1,2,n}|^2} & \le C \E \Big[ \Big| \frac{1}{n} \sum_{i=1}^{\lfloor nt\rfloor }\sum_{j=1}^{\lfloor nt\rfloor } \Indi{\{j \ge i +m + 1\}}  f_n( {\bm{X}}_{\frac{i-1}{n}},\mathcal{I}_{i,n}) \left(f_n ( {\bm{X}}_{\frac{j-1}{n}} ,
\mathcal{I}_{j,n}) - f_n ( \dot{\bm{X}}_{\frac{j-1}{n}}[i] ,
\dot{\mathcal{I}}_{j,n}[i])  \right) \Big|^2 \Big] \nonumber \\
& = \frac{1}{n^2} \sum_{1 \le i_1, i_2, j_1, j_2 \leq \lfloor nt\rfloor}
 \Indi{\{j_1 \ge i_1 + m + 1, j_2 \ge i_2 + m + 1 \}} \E \Big[ f_n( {\bm{X}}_{\frac{i_1-1}{n}},\mathcal{I}_{i_1,n})f_n( {\bm{X}}_{\frac{i_2-1}{n}},\mathcal{I}_{i_2,n}) \\
 & \times \Big{(}f_n ( {\bm{X}}_{\frac{j_1-1}{n}} ,
\mathcal{I}_{j_1,n}) - f_n ( \dot{\bm{X}}_{\frac{j_1-1}{n}}[i_1] ,
\dot{\mathcal{I}}_{j_1,n}[i_1])  \Big{)} \nonumber \\
& \times 
\Big{(}f_n ( {\bm{X}}_{\frac{j_2-1}{n}} ,
\mathcal{I}_{j_2,n}) - f_n ( \dot{\bm{X}}_{\frac{j_2-1}{n}}[i_2] ,
\dot{\mathcal{I}}_{j_2,n}[i_2])  \Big{)} \Big]. \nonumber
\end{align}
 {Next, by making use of \eqref{eq:1new}, we will show that the non-vanishing terms in the right-hand side of the above sum must satisfy $|j_1-j_2|\leq m$. To show this, we consider the complementary configurations: $j_1 + m < j_2$ and $j_2+m<j_1$. By \eqref{eq:1new}, for all $j> i+m$, the increments $\bm{X}_{j/n}-\bm{X}_{(j-1)/n}$ and $\bm{\dot{X}}_{j/n}{\revch [i]}-\bm{\dot{X}}_{(j-1)/n}{\revch [i]}$ coincide, while under the same regime, it holds that  
 \begin{equation}{\label{eq: 6.2.25}}
   {\bm{X}}_{\frac{j_k-1}{n}} - \dot{\bm{X}}_{\frac{j_k-1}{n}}[i_k] = 
\bm{X}_{\frac{i + m}{n}} - \bm{X}_{\frac{i-1}{n}}
-\bm{\tilde{X}}_{\frac{i + m}{n}} + \bm{\tilde{X}}_{\frac{i-1}{n}}.  
 \end{equation}
Assuming without loss of generality that $j_{1}+m< j_{2}$, we can write
\begin{align*}
& \E\Big[ f_n( {\bm{X}}_{\frac{i_1-1}{n}},\mathcal{I}_{i_1,n})f_n( {\bm{X}}_{\frac{i_2-1}{n}},\mathcal{I}_{i_2,n}) { \Big{(}}f_n ( {\bm{X}}_{\frac{j_1-1}{n}} ,
\mathcal{I}_{j_1,n}) - f_n ( \dot{\bm{X}}_{\frac{j_1-1}{n}}[i_1] ,
\dot{\mathcal{I}}_{j_1,n}[i_1])  \Big{)} \\
& \times \Big{(}f_n ( {\bm{X}}_{\frac{j_2-1}{n}} ,
\mathcal{I}_{j_2,n}) - f_n( \dot{\bm{X}}_{\frac{j_2-1}{n}}[i_2] ,
\dot{\mathcal{I}}_{j_2,n}[i_2])  \Big{)}\Big] \\
& = \E \Big{[} f_n( {\bm{X}}_{\frac{i_1-1}{n}},\mathcal{I}_{i_1,n})f_n( {\bm{X}}_{\frac{i_2-1}{n}},\mathcal{I}_{i_2,n}) \Big{(}f_n ( {\bm{X}}_{\frac{j_1-1}{n}} ,
\mathcal{I}_{j_1,n}) - f_n ( \dot{\bm{X}}_{\frac{j_1-1}{n}}[i_1] ,
{\mathcal{I}}_{j_1,n})  \Big{)} \\
& \times \E\Big[\Big{(}f_n ( {\bm{X}}_{\frac{j_2-1}{n}} ,
\mathcal{I}_{j_2,n}) - f_n ( \dot{\bm{X}}_{\frac{j_2-1}{n}}[i_2] ,
{\mathcal{I}}_{j_2,n})  \Big{)}| \mathcal{F}^{(j_2)}\Big] \Big{]},
\end{align*}
 where $\mathcal{F}^{(j)}$ is defined by \eqref{sialg}. Using the fact that $\bm{X}_{(j_{2}-1)/n},\bm{\dot{X}}_{(j_{2}-1)/n}{\revch [i_2]}$ are independent of $\mathcal{F}^{(j_2)}$, we obtain

 \begin{align}{\label{eq: centered new}}
\E\Big[f_n ( {\bm{X}}_{\frac{j_2-1}{n}} ,
\mathcal{I}_{j_2,n})| \mathcal{F}^{(j_2)}\Big] 
=\E\Big[f_n ( \dot{\bm{X}}_{\frac{j_2-1}{n}}[i_2] ,
\dot{\mathcal{I}}_{j_2,n}[i_2])  | \mathcal{F}^{(j_2)}\Big]=0.
\end{align}
The configuration $j_{2}+m<j_{1}$ can be shown to be equal to zero by an analogous argument.}  The only case it remains to handle is $j_1 - m \le j_2 \le j_1 + m$. Notice that the right hand side of \eqref{eq: square T1} is 
\begin{align}\label{eq: T1 bound 3prel}
& \frac{1}{n^2} \sum_{1 \le i_1, i_2, j_1 \leq \lfloor nt\rfloor} \sum_{j_2 = j_1 - m}^{j_1 + m} \Indi{\{j_1 \ge i_1 +m + 1, j_2 \ge i_2 + m + 1 \}} \E\Big[ f_n( {\bm{X}}_{\frac{i_1-1}{n}},\mathcal{I}_{i_1,n})f_n( {\bm{X}}_{\frac{i_2-1}{n}},\mathcal{I}_{i_2,n}) \\
& \times \Big{(}f_n ( {\bm{X}}_{\frac{j_1-1}{n}} ,\mathcal{I}_{j_1,n}) - f_n ( \dot{\bm{X}}_{\frac{j_1-1}{n}}[i_1] ,
\dot{\mathcal{I}}_{j_1,n}[i_1])  \Big{)}  \Big{(}f_n ( {\bm{X}}_{\frac{j_2-1}{n}} ,
\mathcal{I}_{j_2,n}) - f_n ( \dot{\bm{X}}_{\frac{j_2-1}{n}}[i_2] ,
\dot{\mathcal{I}}_{j_2,n}[i_2])  \Big{)}\Big].\nonumber
\end{align}
{Taking into consideration the boundedness of $g$ and its first derivative, for every $i_1,i_2,j_1,j_2\in\N$, every   choice of $u=1,2$ and every  collection of vectors $\bm{\varepsilon}^{\ell,1},\bm{\varepsilon}^{\ell,2}$ with $\ell=0,\dots, {m+1}$, having  form $\bm{\varepsilon}^{\ell,u}=(\varepsilon_{1}^{\ell,u},\dots, \varepsilon_{d}^{\ell,u})$, we can write
\begin{align*}
|g(\bm{X}_{\frac{i_u-1}{n}}+\bm{\varepsilon}^{0,u},\mathcal{I}_{i_u,n}+(\bm{\varepsilon}^{1,u},\dots,\bm{\varepsilon}^{{m+1},u}))
-g(\bm{X}_{\frac{i_u-1}{n}},\mathcal{I}_{i_u,n})|
  \leq (2\|g\|_{\infty})\wedge\|Dg\|_{\infty}\sum_{i=1}^d\sum_{j=0}^{{m+1}}|\varepsilon_{i}^{j,u}|,
\end{align*}
with a similar bound being satisfied by replacing $g$ by  $f_{n}$. {Taking $\bm{\varepsilon}^{0,u}=\bm{\dot{X}}_{\frac{j_u-1}{n}}[i_u]-\bm{X}_{\frac{j_u-1}{n}}$ and $(\bm{\varepsilon}^{1,u},\dots, \bm{\varepsilon}^{{m+1},u})=\dot{\mathcal{I}}_{j_u,n}[i_u]-{\mathcal{I}}_{j_u,n}$, with $i_1,i_2,j_1, j_2$ as above, we get}
\begin{align*}
|\bm{\varepsilon}^\ell|
  &\leq a_n \sum_{u=1,2}{\sum_{l=0}^{{m+1}}}(\|\bm{X}_{\frac{i_u+l}{n}}-\bm{X}_{\frac{i_u+l-1}{n}}\|
  +\|\bm{\tilde{X}}_{\frac{i_u+l}{n}}-\bm{\tilde{X}}_{\frac{i_u+l-1}{n}}\|),
\end{align*}
for $\ell=1,\dots,{m+1}$. This allows us to conclude that \eqref{eq: T1 bound 3prel} is upper bounded by
$\mathcal{E}_{1}+\mathcal{E}_2$, where
\begin{align*}
\mathcal{E}_{1}
    &:= \frac{C}{n^2} \sum_{1 \le i_1, i_2, j_1 \leq \lfloor nt\rfloor} \sum_{j_2 = j_1 - m}^{j_1 + m} {\sum_{0\leq l_1,l_2\leq{m+1}}}\E
\left[\Big(1 \land \|{\bm{X}}_{\frac{i_1{+l_1-1}}{n}} - {\bm{X}}_{\frac{i_1-1}{n}}\|\Big)\Big(1 \land \|{\bm{X}}_{\frac{i_2{+l_2-1}}{n}} - {\bm{X}}_{\frac{i_2-1}{n}}\|\Big) \right] \\
\mathcal{E}_{2}
    &:=\frac{C}{n^2} \sum_{1 \le i_1, i_2, j_1 \leq \lfloor nt\rfloor} \sum_{j_2 = j_1 - m}^{j_1 + m} {\sum_{l_1,l_2=0}^{ {m+1}}}\E
\left[\Big(1 \land \|{\bm{X}}_{\frac{i_1{+l_1-1}}{n}} - {\bm{X}}_{\frac{i_1-1}{n}}\|\Big)\right] \E\left[\Big(1 \land \|{\bm{X}}_{\frac{i_2{+l_2-1}}{n}} - {\bm{X}}_{\frac{i_2-1}{n}}\|\Big) \right] .
\end{align*}
The term $\mathcal{E}_1$ can be upper bounded by
\begin{align*}
& \frac{C}{n^2} \sum_{1 \le i_1, i_2, j_1 \leq \lfloor nt\rfloor} \sum_{j_2 = j_1 - m}^{j_1 + m} {\sum_{0\leq l_1,l_2\leq {m+1}}}\E
\left[\Big(1 \land \|{\bm{X}}_{\frac{i_1{+l_1-1}}{n}} - {\bm{X}}_{\frac{i_1-1}{n}}\|\Big)\Big(1 \land \|{\bm{X}}_{\frac{i_2{+l_2-1}}{n}} - {\bm{X}}_{\frac{i_2-1}{n}}\|\Big) \right] \\
& \le \frac{C}{n} \left( {\sum_{1\leq i_1,i_2\leq   nt }  \Indi{\{|i_2-i_1|\leq m\}}{\sum_{0\leq l_1,l_2\leq {m+1}}} \E\left[\left(1 \land \|{\bm{X}}_{\frac{i_1+l_1-1}{n}} - {\bm{X}}_{\frac{i_1-1}{n}}\|\right)\left(1 \land \|{\bm{X}}_{\frac{i_2+l_2-1}{n}} - {\bm{X}}_{\frac{i_2-1}{n}}\|\right)\right]} \right.\nonumber\\
&\left.+ \sum_{1 \le i_1, i_2 \leq \lfloor nt\rfloor} {\Indi{\{|i_2 -i_1|>m\}} {\sum_{0\leq l_1,l_2\leq {m+1}}}\E\left[ \left(1 \land \|{\bm{X}}_{\frac{i_1+l_1-1}{n}} - {\bm{X}}_{\frac{i_1-1}{n}}\|\right) \left(1 \land \|{\bm{X}}_{\frac{i_2+l_2-1}{n}} - {\bm{X}}_{\frac{i_2-1}{n}}\|\right) \right]}\right)\nonumber.
\end{align*}
By writing 
$$\left(1 \land \|{\bm{X}}_{\frac{i+l-1}{n}} - {\bm{X}}_{\frac{i-1}{n}}\|\right)
  \leq \sum_{\ell=1}^l\left(1 \land \|{\bm{X}}_{\frac{i+\ell}{n}} - {\bm{X}}_{\frac{i+\ell-1}{n}}\|\right),$$
for $i,l\geq0$ arbitrary, we deduce that 
\begin{align}{\label{eq: T1 bound 3}}
|T_{1,2,n}|^2 
    & \le \frac{C}{n} \left( {\sum_{1\leq i_1,i_2\leq   nt }  \Indi{\{|i_2-i_1|\leq m\}} \E\left[\left(1 \land \|{\bm{X}}_{\frac{i_1-1}{n}} - {\bm{X}}_{\frac{i_1-1}{n}}\|\right)\left(1 \land \|{\bm{X}}_{\frac{i_2-1}{n}} - {\bm{X}}_{\frac{i_2-1}{n}}\|\right)\right]} \right.\\
&\left.+ \sum_{1 \le i_1, i_2 \leq  nt} {\Indi{\{|i_2 -i_1|>m\}}\E\left[ \left(1 \land \|{\bm{X}}_{\frac{i_1-1}{n}} - {\bm{X}}_{\frac{i_1-1}{n}}\|\right) \left(1 \land \|{\bm{X}}_{\frac{i_2-1}{n}} - {\bm{X}}_{\frac{i_2-1}{n}}\|\right) \right]} \right)+\mathcal{E}_{2}\nonumber.
\end{align}
}
The first term on the right is {of the order $(1/n)(1+\Indi{\{\alpha=2\}}\log(n))$  due to \eqref{eq:exponewedgex1 lemma}. }
For the second term we use the independence of the increments of $\bm{X}$, which together with Lemma \ref{eq:ineqoneminx}, gives
\begin{align}\label{eq:minimumindep}
\E[(1 \land \|{\bm{X}}_{\frac{i_1}{n}} - {\bm{X}}_{\frac{i_1-1}{n}}\|) (1 \land \|{\bm{X}}_{\frac{i_2}{n}} - {\bm{X}}_{\frac{i_2-1}{n}}\| )]
  \le \E[1\wedge \|{\bm{X}}_{\frac{1}{n}}\|]^2.
\end{align}
By Lemma \ref{eq:ineqoneminx}, we thus obtain the bound {
\begin{align*}
|T_{1,2,n}|^2 
\le \mathcal{E}_2+
C \left( n^{1-\frac{2}{\alpha}}+n^{-1}(1+\Indi{\{\alpha=1\}}\log(n)^2+\Indi{\{\alpha=2\}}\log(n)) \right).
\end{align*}
The term $\mathcal{E}_2$ can be easily handled by combining  \eqref{eq:minimumindep} and Lemma \ref{eq:ineqoneminx} as before, yielding the final inequality

\begin{align}\label{eq:T1boundidk}
|T_{1,2,n}|^2 
\le 
C \left( n^{1-\frac{2}{\alpha}}+n^{-1}(1+\Indi{\{\alpha=1\}}\log(n)^2+\Indi{\{\alpha=2\}}\log(n)) \right).
\end{align}
}

\noindent We now deal with $T_{1,1,n}$. We begin writing
\begin{align*}
|T_{1,1,n}|
  &\le \frac{1}{\sqrt{n}} \sum_{i=1}^{\lfloor nt\rfloor}  \E \Big{[} \Big{|}   h''_V ({Z}_{t}^{(n)}) - h''_V (\dot{Z}_{t}^{(n)}[i])\Big{|} \Big{|}f_n( {\bm{X}}_{\frac{i-1}{n}} , \mathcal{I}_{i,n}) \Big{|} \\
& \times \Big| \frac{1}{\sqrt{n}} \sum_{j=1}^{\lfloor nt\rfloor} \Indi{\{j \ge i +m + 1\}} (f_n ( {\bm{X}}_{\frac{j-1}{n}} ,
\mathcal{I}_{j,n}) - f_n ( \dot{\bm{X}}_{\frac{j-1}{n}}[i] ,
\dot{\mathcal{I}}_{j,n}[i])) \Big | \Big{]}.
\end{align*}
Using Cauchy-Schwarz inequality and the boundedness of $g$, we get 
\begin{align}{\label{eq: T11 start}}
|T_{1,1,n}|
& \leq \frac{C}{\sqrt{n}} \sum_{i=1}^{\lfloor nt\rfloor } \E \Big{[} \Big{|}   h''_V ({Z}_{t}^{(n)}) - h''_V (\dot{Z}_{t}^{(n)}[i])\Big{|}^2 \Big{]}^{\frac{1}{2}} \\
& \times \E\Big{[} \Big{|}  \frac{1}{\sqrt{n}} \sum_{j=1}^{\lfloor nt\rfloor } \Indi{\{j \ge i +m +  1\}} \Big(f_n ( {\bm{X}}_{\frac{j-1}{n}} ,
\mathcal{I}_{j,n}) - f_n ( \dot{\bm{X}}_{\frac{j-1}{n}}[i] ,
\dot{\mathcal{I}}_{j,n}[i])  \Big)\Big{|}^2 \Big{]}^{\frac{1}{2}}. \nonumber
\end{align}
Regarding the first term in the product appearing above, Taylor development and the boundedness of the third derivatives of $h_V$ ensures that 
\begin{equation}{\label{eq: 0.5 deriv hV}}
\E \Big{[} \Big{|}    h''_V ({Z}_{t}^{(n)}) - h''_V (\dot{Z}_{t}^{(n)}[i])\Big{|}^2 \Big{]}^{\frac{1}{2}} \le C\E \Big{[} \Big{|}    {Z}_{t}^{(n)} - \dot{Z}_{t}^{(n)}[i]\Big{|}^2 \Big{]}^{\frac{1}{2}}{\rev \le C n^{- \frac{1}{2}}(1+\Indi{\{\alpha=2\}}\log(n))^\frac{1}{2}},
\end{equation}
{\rev thanks to Lemma \ref{lemma: L2 z-zdot}.}
To conclude the study of $T_{1, 1}$ we are left to bound the second term in the product on the right-hand side of \eqref{eq: T11 start}. We proceed in a similar way as in \eqref{eq: square T1}. First we write
\begin{align}{\label{eq: square T11}}
&\E \Big[ \Big|  \frac{1}{\sqrt{n}} \sum_{j=1}^{\lfloor nt\rfloor } \Indi{\{j \ge i +m +  1\}} \Big(f_n ( {\bm{X}}_{\frac{j-1}{n}} ,
\mathcal{I}_{j,n}) - f_n ( \dot{\bm{X}}_{\frac{j-1}{n}}[i] ,
\dot{\mathcal{I}}_{j,n}[i]) \Big ) \Big|^2 \Big] \nonumber \\
 & = \frac{1}{n} \sum_{1 \le j_1, j_2 \leq \lfloor nt\rfloor}
 \Indi{\{j_1, j_2 \ge i + m+ 1 \}} \E \Big[\Big{(}f_n ( {\bm{X}}_{\frac{j_1-1}{n}} ,
\mathcal{I}_{j_1,n}) - f_n ( \dot{\bm{X}}_{\frac{j_1-1}{n}}[i] ,
\dot{\mathcal{I}}_{j_1,n}[i])  \Big{)} \\
& \times \Big{(}f_n ( {\bm{X}}_{\frac{j_2-1}{n}} ,
\mathcal{I}_{j_2,n}) - f_n ( \dot{\bm{X}}_{\frac{j_2-1}{n}}[i] ,
\dot{\mathcal{I}}_{j_2,n}[i])  \Big{)} \Big]. \nonumber
\end{align}
As remarked below \eqref{eq: square T1}, the variables {\revch $f_n ( {\bm{X}}_{\frac{j-1}{n}} ,
\mathcal{I}_{j,n})$ and $f_n ( \dot{\bm{X}}_{\frac{j-1}{n}}[i] ,
\dot{\mathcal{I}}_{j,n}[i])$} are conditionally centered given $\mathcal{F}_{\frac{j-1}{n}}$. Hence, as both $j_1$ and $j_2$ are larger than $i + 1 + m$, if $j_1 + m < j_2$ or $j_2 + m < j_1$, the quantity above is equal to $0$.
Thus, the only configuration for which it is not equal to zero is $j_1 - m \le j_2 \le j_1 + m$. In this case, the mean value theorem together with the boundedness of the partial derivatives of $g$, and so of the partial derivatives of $f$, implies that \eqref{eq: square T11} is upper bounded by
\begin{align*}
&\frac{1}{n} \sum_{1 \le j_1 \leq \lfloor nt\rfloor}
\sum_{j_2 = j_1 - m}^{j_1 + m} \Indi{\{j_1, j_2 \ge i +m + 1 \}} \E\left[(1 \land \|{\bm{X}}_{\frac{i}{n}} - {\bm{X}}_{\frac{i-1}{n}}\| )^2\right],
\end{align*}
where we have also used that ${\bm{X}}_{\frac{j-1}{n}} - \dot{\bm{X}}_{\frac{j-1}{n}}[i]$ {\revch is as in \eqref{eq: 6.2.25}} {\rev for $j \ge i + m + 1$}. {\rev By Lemma \ref{eq:ineqoneminx}, this term is upper bounded by  $Cn^{-1}(1+\Indi{\{\alpha=2\}}\log(n))$ and by $Cn^{-1}$ in the presence of $\textbf{H}_3$}. Applying the controls \eqref{eq: bound increment Z} and \eqref{eq: square T11} in \eqref{eq: T11 start}, we obtain that 
{
\begin{align}\label{eq:T1boundidkprime}
|T_{1,1,n}|
  &\leq Cn^{-1/2}(1+\Indi{\{\alpha=2\}}\log(n))
\end{align}
and under the presence of {$\textbf{H}_3$}, 
\begin{align*}
{|T_{1,1,n}|}
  &\leq Cn^{-1/2}.
\end{align*}
The inequalities \eqref{eq:T1boundidk} and \eqref{eq:T1boundidkprime} give the following bound
\begin{align}\label{Tqfinalboundfirstregime}
|T_{1,n}| 
\le 
C (n^{\frac{1}{2}-\frac{1}{\alpha}}+n^{-1/2}(1+\Indi{\{{\alpha\in\{1,2\}}\}}\log(n))).
\end{align}

}

\subsubsection{Analysis of {$T_{3,n}$}} 
{\rev Recall that} {$T_{3,n}$} is given by
\begin{align}{\label{eq: def T3}}
T_{3,n}& = {\rev \frac{1}{\sqrt{n}} \sum_{i=1}^{\lfloor nt \rfloor} \E \Big{[} f_n ( {\bm{X}}_{\frac{i-1}{n}} ,
\mathcal{I}_{i,n})  h'''_V (\dot{Z}_{t}^{(n)}[i] + \tau  (Z_{t}^{(n)} - \dot{Z}_{t}^{(n)}[i])) (Z_{t}^{(n)} - \dot{Z}_{t}^{(n)}[i])^2  \Big{]}.}
\end{align}
{In order to handle the right-hand side, we proceed as follows. First we observe that}  $\|h'''_V\|_{\infty}<\infty$ ensures that 
\begin{align*}
{|T_{3,n}|} 
  &\le C\sqrt{n} \max_{i} \E\left[(Z_{t}^{(n)} - \dot{Z}_{t}^{(n)}[i])^2\right].
\end{align*}
 From \eqref{eq: bound increment Z} it directly follows
\begin{equation}{\label{eq: end T3}}
{|T_{3,n}|}  \le C n^{-\frac{1}{2}}\left(1+\Indi{\{\alpha=2\}}\log(n)\right).
\end{equation}

\subsubsection{Analysis of ${T_{4,n}} $} 
Regarding the last term, we notice that
\begin{align}{\label{eq: new 1}}
| {T_{4,n}}| \le \frac{1}{n} \sum_{i=1}^{\lfloor n t \rfloor} \Big|\E \Big{[} \Big( h''_V (Z_{t}^{(n)})- h''_V (\dot{Z}_{t}^{(n)}[i])\Big) f_n ( {\bm{X}}_{\frac{i-1}{n}} ,\mathcal{I}_{i,n}) \sum_{j = i - m}^{i + m} f_n ( {\bm{X}}_{\frac{j-1}{n}} ,
\mathcal{I}_{j,n})  \Big{]} \Big|.   
\end{align}
We observe that from Taylor development and $\|h'''_V\|_{\infty}<\infty$, we clearly have that
\begin{equation}{\label{eq: z}}
\E \Big{[} \Big{|}    h''_V ({Z}_{t}^{(n)}) - h''_V (\dot{Z}_{t}^{(n)}[i])\Big{|}^2 \Big{]}^{\frac{1}{2}} \le C\E \Big{[} \Big{|}    {Z}_{t}^{(n)} - \dot{Z}_{t}^{(n)}[i]\Big{|}^2 \Big{]}^{\frac{1}{2}}.
\end{equation}
Then, the boundedness of $f_n$ together with Lyapunov inequality and equations \eqref{eq: 0.5 deriv hV} and \eqref{eq: z} above provide
\begin{equation}{\label{eq: end T4}}
|{T_{4,n}}| \le C n^{- \frac{1}{2}}\left(1 + \log (n) \Indi{ \{ \alpha = 2 \} }\right)^\frac{1}{2} \le C n^{- \frac{1}{2}}\left(1 + \log (n) \Indi{ \{ \alpha = 2 \} }\right).    
\end{equation}

{\noindent Relations \eqref{Tqfinalboundfirstregime}, \eqref{eq: end T3} and \eqref{eq: end T4}  complete the proof of Proposition \ref{prop: T negl} in regime one.}

\qed

\subsection{Case $\alpha \in(1,2]$ under condition {$\bf{H_2}(\alpha)$}}
We now show that under the assumption that $\alpha\in(1,2]$ and {$\bm{H}_2(\alpha)$} holds, the convergence rate can be improved to {$n^{-1/2}(1+\Indi{\{\alpha=2\}} \log(n)^{3/2})$} (and the logarithm can be removed if {$\bm{H}_3$} holds true).

\subsubsection{Analysis of  { $T_{1,n}$}} 
{ Observe that if $i,j\in\N$ are such that $j\geq i+m+1$, then $\dot{\mathcal{I}}_{j,n}[i]=\mathcal{I}_{j,n}$. In particular, } $T_{1,n}$ equals
\begin{align*}
& \frac{1}{n} \sum_{i=1}^{\lfloor nt\rfloor}\sum_{j=1}^{\lfloor nt \rfloor} \Indi{\{j \ge i +m + 1\}}\E \Big{[} {f_n\Big( {\bm{X}}_{\frac{i-1}{n}} ,
\mathcal{I}_{i,n}\Big)}   h''_V (\dot{Z}_{t}^{(n)}[i])
\Big(f_n ( {\bm{X}}_{\frac{j-1}{n}} ,
{\dot{\mathcal{I}}_{j,n}[i]}) - f_n ( \dot{\bm{X}}_{\frac{j-1}{n}}[i] ,
\dot{\mathcal{I}}_{j,n}[i])  \Big)  \Big{]}.
\end{align*}
{\rev We recall that, according to \eqref{eq:1new}, 
it is ${\bm{X}}_{\frac{j-1}{n}} - \dot{\bm{X}}_{\frac{j-1}{n}}[i] = {\bm{X}}_{\frac{i + m}{n}} - {\bm{X}}_{\frac{i-1}{n}}  - \bm{\tilde{X}}_{\frac{i + m}{n}} + \bm{\tilde{X}}_{\frac{i - 1}{n}}$ for $j \ge i + m + 1$.}
{Then, by a  Taylor development, 
\begin{multline}\label{eq:TaylorexpansionT1}
f_n ( {\bm{X}}_{\frac{j-1}{n}} ,{\dot{\mathcal{I}}_{j,n}[i]}) - f_n ( \dot{\bm{X}}_{\frac{j-1}{n}}[i] ,\dot{\mathcal{I}}_{j,n}[i])\\
\begin{aligned} 
  &= \sum_{r = 1}^d \frac{\partial f_n}{\partial x_r}  ( \dot{\bm{X}}_{\frac{j-1}{n}}[i] ,\dot{\mathcal{I}}_{j,n}[i]){\hat{\mathcal{J}}}_{i,n}^{(r)}\\
  &+ \int_0^1 \sum_{r_1, r_2 = 1}^d \frac{\partial^2f_n}{\partial x_{r_1} \partial  x_{r_2} }   {\Big(\dot{\bm{X}}_{\frac{j-1}{n}}[i] {+}  \lambda (\mathcal{I}_{i,n} -\tilde{\mathcal{I}}_{i,n})} ,\dot{\mathcal{I}}_{j,n}[i] \Big){\hat{\mathcal{J}}_{i,n}^{(r_1)} \hat{\mathcal{J}}_{i,n}^{(r_2)}} d\lambda,
\end{aligned}
\end{multline}
{where 
\begin{align*}
\hat{\mathcal{J}}_{i,n}^{(r)}
  &:={\mathcal{J}}_{i,n}^{(r)}-\tilde{\mathcal{J}}_{i,n}^{(r)},
\end{align*}
for ${\revch \tilde{\mathcal{J}}_{i,n}^{ (r)}}$ defined analogously to ${\revch \mathcal{J}_{i,n}^{(r)}}$, but with the process $\bm{X}$ replaced by its independent copy $\bm{\tilde{X}}$.
}

In particular, {$T_{1,n}=T_{1,3,n}+T_{1,4,n}$}, where 
\begin{align*}
{T_{1,3,n}}
  &:=\frac{1}{n} \sum_{r=1}^d\sum_{1\leq i,j\leq \lfloor nt \rfloor} \Indi{\{j \ge i +m + 1\}}\\
  &\times \E \Big{[} f_n( {\bm{X}}_{\frac{i-1}{n}} ,
\mathcal{I}_{i,n})   h''_V \left(\dot{Z}_{t}^{(n)}[i]\right)
\frac{\partial f_n}{\partial x_r}  \left({\bm{\dot{X}}_{\frac{j-1}{n}}[i] ,\dot{\mathcal{I}}_{j,n}[i]} \right){\hat{\mathcal{J}}_{i,n}^{(r)}}  \Big{]}
\end{align*} 
and 
\begin{align*}
{T_{1,4,n}}
  &:=\frac{1}{n} \sum_{i=1}^{\lfloor nt\rfloor}\sum_{j=1}^{\lfloor nt \rfloor} \Indi{\{j \ge i +m + 1\}}\E \Big{[} f_n( {\bm{X}}_{\frac{i-1}{n}} ,
\mathcal{I}_{i,n})   h''_V (\dot{Z}_{t}^{(n)}[i])\\
&\times 
\int_0^1 \sum_{r_1, r_2 = 1}^d \frac{\partial^2f_n}{\partial x_{r_1} \partial  x_{r_2} }   {\Big(\dot{\bm{X}}_{\frac{j-1}{n}}[i] {+}   \lambda (\mathcal{I}_{i,n} - \tilde{\mathcal{I}}_{i,n} )},\dot{\mathcal{I}}_{j,n}[i] \Big){\hat{\mathcal{J}}_{i,n}^{(r_1)} \hat{\mathcal{J}}_{i,n}^{(r_2)}} d\lambda  \Big{]}.
\end{align*} 

Note that \( T_{1,3,n} = 0 \). This holds because \(\dot{Z}_{t}^{(n)}[i]\), \(\bm{\dot{X}}_{\frac{j-1}{n}}[i]\), and \(\dot{\mathcal{I}}_{j,n}[i]\) are, by construction, all measurable with respect to \(\mathcal{F}^{(i)}\). Moreover, we have
$$\E \Big[ f_n( \bm{X}_{\frac{i-1}{n}} , \mathcal{I}_{i,n}) \hat{\mathcal{J}}_{i,n}^{(r)} | \mathcal{F}^{(i)} \Big] = \E \Big[ f_n( \bm{X}_{\frac{i-1}{n}} , \mathcal{I}_{i,n}) \mathcal{J}_{i,n}^{(r)} | \mathcal{F}^{(i)} \Big] - \E \Big[ f_n( \bm{X}_{\frac{i-1}{n}} , \mathcal{I}_{i,n}) \tilde{\mathcal{J}}_{i,n}^{(r)} | \mathcal{F}^{(i)} \Big].$$
The first term is zero due to \(\bm{H}_2(\alpha)\). Since \(\bm{\tilde{X}}\) is an independent copy of \(\bm{X}\), the independence of \(f_n( \bm{X}_{\frac{i-1}{n}}, \mathcal{I}_{i,n})\) and \(\tilde{\mathcal{J}}_{i,n}^{(r)}\) implies
$$\E \Big[ f_n( \bm{X}_{\frac{i-1}{n}} , \mathcal{I}_{i,n}) \tilde{\mathcal{J}}_{i,n}^{(r)} | \mathcal{F}^{(i)} \Big] = \E \Big[ f_n( \bm{X}_{\frac{i-1}{n}} , \mathcal{I}_{i,n}) | \mathcal{F}^{(i)} \Big] \E \Big[ \tilde{\mathcal{J}}_{i,n}^{(r)} | \mathcal{F}^{(i)} \Big] = 0,$$
the last equality follows from \eqref{eq: 5.8.5}.
Therefore, it suffices to handle the term $T_{1,4,n}$. Define 
\begin{align}\label{eq:Uijndef}
U_{i,j}^{(n)}
  &:=\int_0^1 \sum_{r_1, r_2 = 1}^d \frac{\partial^2f_n}{\partial x_{r_1} \partial  x_{r_2} }   {\Big(\dot{\bm{X}}_{\frac{j-1}{n}}[i] {+}  \lambda (\mathcal{I}_{i,n}-\tilde{\mathcal{I}}_{i,n})} ,\dot{\mathcal{I}}_{j,n}[i] \Big){\hat{\mathcal{J}}_{i,n}^{(r_1)} \hat{\mathcal{J}}_{i,n}^{(r_2)}} d\lambda,
\end{align}
so we can write 
\begin{align*}
{T_{1,4,n}}
  &=\frac{1}{n} \sum_{i=1}^{\lfloor nt\rfloor}\sum_{j=1}^{\lfloor nt \rfloor} \Indi{\{j \ge i +m + 1\}}\E \Big{[} f_n( {\bm{X}}_{\frac{i-1}{n}} ,
\mathcal{I}_{i,n})   h''_V (\dot{Z}_{t}^{(n)}[i])U_{i,j}^{(n)}  \Big{]}.
\end{align*}
{In analogy to the definition of $\dot{\bm{X}}[i]=(\dot{\bm{X}}_{t}[i]\ ;\ t\geq0)$ in \eqref{eq:Xdotdef} we now introduce the process
$\dot{\bm{X}}[i, j]=(\dot{\bm{X}}_{t}[i,j]\ ;\ t\geq0)$ with $\dot{\bm{X}}_{t}[i,j]=(\dot{{X}}_{t}^{(1)}[i, j],\dots, \dot{{X}}_{t}^{(d)}[i,j])$, where
\begin{align}
\dot{{X}}_{t}^{(\ell)}[i, j]
&:=\int_0^{t}\Indi{\R\backslash([\frac{i-1}{n},\frac{i + m}{n}]\cup [\frac{j-1}{n},\frac{j + m}{n}])}(s)dX_{s}^{(\ell)}\nonumber\\
&{+\int_0^{t}\Indi{[\frac{i-1}{n},\frac{i + m}{n}]\cup [\frac{j-1}{n},\frac{j + m}{n}])}(s)d\tilde{X}_{s}^{(\ell)},}
\end{align}
as well as
\begin{align*}
\dot{Z}_{t}^{(n)}[i,j]
  &:=\frac{1}{\sqrt{n}}\sum_{k=1}^{\lfloor nt\rfloor} f_n\left( \dot{\bm{X}}_{\frac{k-1}{n}}[i,j],
\dot{\mathcal{I}}_{k,n}[i,j]\right).
\end{align*}}
We obtain {$T_{1,n}=T_{1,4,n}=T_{1,5,n}+T_{1,6,n}$}, where 
\begin{align*}
{T_{1,5,n}}
  &=\frac{1}{n} \sum_{i=1}^{\lfloor nt\rfloor}\sum_{j=1}^{\lfloor nt \rfloor} \Indi{\{j \ge i +m + 1\}}\E \Big{[} f_n( {\bm{X}}_{\frac{i-1}{n}} ,
\mathcal{I}_{i,n})   h''_V (\dot{Z}_{t}^{(n)}[i,j])U_{i,j}^{(n)}  \Big{]},\\
{T_{1,6,n}}
  &=\frac{1}{n} \sum_{i=1}^{\lfloor nt\rfloor}\sum_{j=1}^{\lfloor nt \rfloor} \Indi{\{j \ge i +m + 1\}}\E \Big{[} f_n( {\bm{X}}_{\frac{i-1}{n}} ,
\mathcal{I}_{i,n})   \Big(h''_V (\dot{Z}_{t}^{(n)}[i])-h''_V (\dot{Z}_{t}^{(n)}[i,j])\Big)U_{i,j}^{(n)}  \Big{]}.
\end{align*}
From \eqref{eq:TaylorexpansionT1} one can write 
\begin{equation}{\label{eq: Uij 6.20.5}}
   U_{i,j}^{(n)} = f_n ( {\bm{X}}_{\frac{j-1}{n}} ,\mathcal{I}_{j,n}) - f_n ( \dot{\bm{X}}_{\frac{j-1}{n}}[i] ,\dot{\mathcal{I}}_{j,n}[i]) - \sum_{r = 1}^d \frac{\partial f_n}{\partial x_r}  ( \dot{\bm{X}}_{\frac{j-1}{n}}[i] ,\dot{\mathcal{I}}_{j,n}[i]) \hat{\mathcal{J}}_{i,n}^{(r)}. 
\end{equation}
Then, using the fact that $g$ and its derivatives are bounded, {\rev the identity \eqref{eq: Uij 6.20.5} implies  that $|U_{i,j}^{(n)}|\leq C(1+ \left \| \hat{\mathcal{J}}_{i,n} \right \|)$.} In addition, using the triangle inequality and the boundedness of the derivatives of order two of $g$, we can easily deduce {\rev from the definition of $U_{i,j}^{(n)}$ as in \eqref{eq:Uijndef} that $|U_{i,j}^{(n)}|\leq C \left \| \hat{\mathcal{J}}_{i,n} \right \|^2$}, thus yielding the inequality
\begin{align}\label{eq:Uijnbound}
|U_{i,j}^{(n)}|
  &\leq C(1+\left \| \hat{\mathcal{J}}_{i,n} \right \|)\wedge \left \| \hat{\mathcal{J}}_{i,n} \right \|^2.
\end{align}
Since $\alpha>1$, we conclude that the term $U_{i,j}^{(n)}$ is integrable. 
Let us introduce 
\begin{align*}
\mathcal{F}^{(i,j)}
  &:= \sigma \big({\bm{X}}_{\frac{k}{n}} - {\bm{X}}_{\frac{k-1}{n}},{\bm{\tilde{X}}}_{s} - {\bm{\tilde{X}}}_{\frac{h-1}{n}}; \, k \neq i, i+ 1, ... , i + m, j, j+ 1, ... , j + m\\ 
  &\ \ \ \ \ \ \ \ \ \text{ and } h=i, i+ 1, ... , i + m, j, j+ 1, ... , j + m,\text{  } s\in[(h-1)/n,h/n]\big)	
\end{align*}
and observe that the partial derivatives of $f_n$ are conditionally centered given $\mathcal{F}^{(i,j)}$. Indeed, due to the condition $j\geq i+m+1$, we have
$$\E\left[\frac{\partial^2f_n}{\partial x_{r_1} \partial x_{r_2}} \Big(\dot{\bm{X}}_{\frac{j-1}{n}}[i] {+} { \lambda (\mathcal{I}_{i,n}-\tilde{\mathcal{I}}_{i,n})},\dot{\mathcal{I}}_{j,n}[i] \Big)| {\mathcal{F}^{(i,j)}}\right] = 0.$$
As a consequence,  as $\hat{\mathcal{J}}_{i,n}^{(r_1)}$ and $\hat{\mathcal{J}}_{i,n}^{(r_2)}$ are measurable with respect to $\mathcal{F}^{(i,j)}$ for $j\geq i+m+1$, we obtain $\E[U_{i,j}^{(n)}| { \mathcal{F}^{(i,j)}}] = 0$. {In addition,  again by the condition $j\geq i+m+1$, it holds that ${\bm{X}}_{\frac{i-1}{n}},  \mathcal{I}_{i,n}$   and $ \dot{Z}_{t}^{(n)}{[i,j]}$ are $\mathcal{F}^{(i,j)}$-measurable, and consequently}
\begin{align*}
\E \Big{[} f_n( {\bm{X}}_{\frac{i-1}{n}} ,
\mathcal{I}_{i,n})   h''_V (\dot{Z}_{t}^{(n)}{[i,j]})U_{i,j}^{(n)}  \Big{]}
  &{=\E \Big{[} f_n( {\bm{X}}_{\frac{i-1}{n}} ,
\mathcal{I}_{i,n})   h''_V (\dot{Z}_{t}^{(n)}{[i,j]})\E[U_{i,j}^{(n)}\ |\ \mathcal{F}^{(i,j)}]  \Big{]}=0},
\end{align*}
which yields {$T_{1,5,n}=0$}. It thus suffices to handle the term {$T_{1,6,n}$}. To this end, we use the boundedness of $g$ and $h'''_V$, as well as inequality \eqref{eq:Uijnbound}, to deduce that 
\begin{align*}
{|T_{1,6,n}|}
  &\leq  \frac{C}{n} \sum_{i=1}^{\lfloor nt\rfloor}\sum_{j=1}^{\lfloor nt \rfloor} \Indi{\{j \ge i +m + 1\}}\E \Big{[} \Big|(\dot{Z}_{t}^{(n)}[i]- \dot{Z}_{t}^{(n)}[i,j]){\rev \Big(}(1+\left \| \hat{\mathcal{J}}_{i,n} \right \|)\wedge \left \| \hat{\mathcal{J}}_{i,n} \right \|^2 {\rev \Big) }\Big|\Big{]}.
\end{align*}
Since the variables $\dot{Z}_{t}^{(n)}[i],\dot{Z}_{t}^{(n)}[i,j]$ are independent of $\mathcal{J}_{i,n}$, we conclude that 
\begin{align*}
{|T_{1,6,n}|}
  &\leq  \frac{C}{n} \sum_{i=1}^{\lfloor nt\rfloor}\sum_{j=1}^{\lfloor nt \rfloor} \Indi{\{j \ge i +m + 1\}}\E \Big{[} \Big| \dot{Z}_{t}^{(n)}[i]- \dot{Z}_{t}^{(n)}[i,j] \Big| \Big{]}\E\Big{[}(1+\left \| \hat{\mathcal{J}}_{i,n} \right \|)\wedge \left \| \hat{\mathcal{J}}_{i,n} \right \|^2\Big{]}.
\end{align*}
Proceeding as {\rev in the proof of Lemma \ref{lemma: L2 z-zdot} one can show that}
\begin{equation}{\label{eq: bound increment Z2}}
{\rev \E \Big{[} \Big| \dot{Z}_{t}^{(n)}[i]- \dot{Z}_{t}^{(n)}[i,j] \Big|^2 \Big{]}\le C n^{-1}(1+\Indi{\{\alpha=2\}}\log(n)) }
\end{equation}
in general and {\rev without the logarithm}
when $\bm{H}_3$ holds. {\rev From Lemma \ref{eq:ineqoneminx} it follows 
$$\E\Big{[}(1+\left \| \hat{\mathcal{J}}_{i,n} \right \|)\wedge \left \| \hat{\mathcal{J}}_{i,n} \right \|^2\Big{]} \le Cn^{-1}(1 + \Indi{\{\alpha=2\}}\log(n)).$$
The combination of these bounds with Lyapunov inequality gives the final bound 
\begin{align*}
{|T_{1,6,n}|}
  &\leq  Cn^{-1/2} (1+\Indi{\{\alpha=2\}}\log(n))^\frac{1}{2} (1+\Indi{\{\alpha=2\}}\log(n)) \le Cn^{-1/2} (1+\Indi{\{\alpha=2\}}\log(n)^\frac{3}{2})
\end{align*}}
in general and 
\begin{align*}
{|T_{1,6,n}|}
  &\leq  Cn^{-1/2} 
\end{align*}
under the presence of $\bm{H}_{3}$}.

\subsubsection{Analysis of $T_{3,n}$ and $T_{4,n}$} 
It is evident that Equations \eqref{eq: end T3} and \eqref{eq: end T4} remain valid within the scope of our hypotheses and that these equations suffice to derive the (enhanced) convergence rate of $n^{-1/2}(1+\Indi{\{\alpha=2\}} \log(n)^{\rev \frac{3}{2}})$ (and the logarithm can be omitted if $\bm{H}_3$ is satisfied). The proof of Proposition \ref{prop: T negl} is then concluded.
\qed

\subsection{{\rev Proof of Lemma \ref{hanalysis}}} \label{sec6.3}
Recall that  $h_{\gamma}$ is the {\rev unique bounded} solution of the Stein equation
\begin{align*}
x h_{\gamma}^{\prime}(x)-\gamma  h_{\gamma}^{\prime\prime}(x)
  &=\psi(x) - \E[\psi(\sqrt{\gamma}N)],
\end{align*}
for a positive real $\gamma\in\R_{+}$. By \cite[Proposition 4.3.2]{NP12}, 
\begin{align*}
h_{\gamma}(x)
  &=\int_0^{\infty} \Big(\E [\psi(\sqrt{\gamma}N)]-{\E}[\psi(e^{-\theta}x+\sqrt{1-e^{-2\theta}}\sqrt{\gamma}N)] \Big)d\theta.
\end{align*}
From here it follows that 
\begin{align*}
h_{\gamma}^{\prime} (x)
  &=-\int_0^{\infty}e^{-\theta}{\E}\left[ \psi^{\prime}(e^{-\theta}x+\sqrt{1-e^{-2\theta}}\sqrt{\gamma}N)\right]d\theta,
\end{align*}
and consequently, if $\tilde{\gamma}>0$ is another positive real,
\begin{align}\label{eq:hprimedifferenceitoin}
 h_{\gamma}^{\prime}(x)- h_{\tilde{\gamma}}^{\prime}(x)
  &=-\int_0^{\infty}e^{-\theta}{\E}\left[\Psi_{x,\theta}(\sqrt{\gamma}N)-\Psi_{x,\theta}(\sqrt{\tilde{\gamma}}N)\right]d\theta,
\end{align}
where $\Psi_{x,\theta}:\R \rightarrow\R$ is the function
\begin{align*}
\Psi_{x,\theta}(y)
  &:= \psi^{\prime}\left(e^{-\theta}x+\sqrt{1-e^{-2\theta}}y\right).
\end{align*}
One can easily check that $\Psi_{x,\theta}\in C^4(\R;\R)$  and 
$$\|\Psi_{x,\theta}^{(4)}\|_{\infty}\leq \|\psi^{(5)}\|_{\infty}<\infty.$$
Thus, using Equation \eqref{eq:hprimedifferenceitoin} and Lemma \ref{lem:distancegaussians}, as well as the fact that{
\begin{align*}
\frac{\partial^{2}\Psi_{x,\theta}}{\partial y^2}(y)
  &= (1-e^{-2\theta})\psi^{(3)}\left(e^{-\theta}x+\sqrt{1-e^{-2\theta}}y\right)\\
\frac{\partial^{4}\Psi_{x,\theta}}{\partial y^4}(y)
  &= (1-e^{-2\theta})^2\psi^{(5)}\left(e^{-\theta}x+\sqrt{1-e^{-2\theta}}y\right),
\end{align*}
we obtain 
\begin{align*}
\left|h_{\gamma}^{\prime}(x)- h_{\tilde{\gamma}}^{\prime}(x)-(\gamma-\tilde{\gamma})\mathcal{V}[\psi](x)\right|
  &\leq C\int_0^{\infty}e^{-\theta}\|\psi^{(5)}\|_{\infty}|\gamma-\tilde{\gamma}|^{2}d\theta,
\end{align*}
which proves the statement of Lemma \ref{hanalysis}. }

\subsection{Proof of Lemma \ref{lemma: L2 z-zdot}} \label{sec6.4}
{Recall the definition of $\mathfrak{g}_n$ and $\mathfrak{g}$ given in \eqref{eq:gnvariancedef} and \eqref{eq:gvariancedef} and that {\revch $\bm{X}_{s} - \dot{\bm{X}}_{s}[i]$ is as in \eqref{eq:1new}. }
Notice that}
\begin{align}\label{eq:VdoV1812}
V-\dot{V}[i]
  &=\int_{\frac{i-1}{n}}^{t}\left(\mathfrak{g}(\bm{X}_{s})-\mathfrak{g}(\dot{\bm{X}}_{s}[i])\right)ds.
\end{align}
{
By Taylor's theorem 
\begin{align*}
|\mathfrak{g}(\bm{X}_{s})-\mathfrak{g}(\dot{\bm{X}}_{s}[i])-\mathfrak{g}^{\prime}(\dot{\bm{X}}_{s}[i])(\bm{X}_{s}-\dot{\bm{X}}_{s}[i])|
  &\leq\frac{1}{2}\|\mathfrak{g}^{\prime\prime}\|_{\infty}\|\bm{X}_{s}-\dot{\bm{X}}_{s}[i]\|^2.
\end{align*}
{\rev Combining the above relation with    \eqref{eq:VdoV1812} and the fact that $\mathfrak{g}^{\prime \prime}$ is bounded, we obtain the statements \eqref{eq:VminusVdotTaylorone} and \eqref{eq:VminusVdotTaylor}.} 

Next, remark that
\begin{align*}
&\E \Big{[} \Big{|}    {Z}_{t}^{(n)} - \dot{Z}_{t}^{(n)}[i]\Big{|}^2 \Big{]} \\
& =  \frac{1}{n} \sum_{j_1, j_2 =1}^{\lfloor nt\rfloor} \E \Big{[} \Big{(}f_n ( {\bm{X}}_{\frac{j_1-1}{n}} ,
\mathcal{I}_{j_1,n}) - f_n ( \dot{\bm{X}}_{\frac{j_1-1}{n}}[i] ,
\dot{\mathcal{I}}_{j_1,n}[i])  \Big{)}\\
&\times \Big{(}f_n ( {\bm{X}}_{\frac{j_2-1}{n}} ,
\mathcal{I}_{j_2,n}) - f_n ( \dot{\bm{X}}_{\frac{j_2-1}{n}}[i] ,
\dot{\mathcal{I}}_{j_2,n}[i])  \Big{)} \Big{]}.
\end{align*}
As before, we examine the different values of $\dot{\bm{X}}_{\frac{j-1}{n}}[i]$ and $\dot{\mathcal{I}}_{j,n}[i]$ as $i$ and $j$ vary.  We split the sum into the instances where $i - m \le j \le i + m$ and $j \ge i+ m+ 1$ to deduce the bound 
\begin{align*}
\E \Big{[} \Big{|}    {Z}_{t}^{(n)} - \dot{Z}_{t}^{(n)}[i]\Big{|}^2 \Big{]} 
  &\leq I_1 + I_2 + I_3,
\end{align*}
where 
\begin{align*}
I_{1}
  &:=\frac{1}{n} \sum_{j_1, j_2= i +m +  1}^{\lfloor nt\rfloor} \Big| \E \Big{[} \Big{(}f_n ( {\bm{X}}_{\frac{j_1-1}{n}} ,
\mathcal{I}_{j_1,n}) - f_n ( \dot{\bm{X}}_{\frac{j_1-1}{n}}[i] ,
\dot{\mathcal{I}}_{j_1,n}[i])  \Big{)} \\
& \times \Big{(}f_n ( {\bm{X}}_{\frac{j_2-1}{n}} ,
\mathcal{I}_{j_2,n}) - f_n ( \dot{\bm{X}}_{\frac{j_2-1}{n}}[i] ,
\dot{\mathcal{I}}_{j_2,n}[i])  \Big{)} \Big{]} \Big|\\
I_{2}
  &:=\frac{2}{n} \sum_{j_1=i +m + 1}^{\lfloor nt\rfloor }\sum_{j_2 = i-m}^{i + m} \Big|\E \Big{[} \Big{(}f_n ( {\bm{X}}_{\frac{j_1-1}{n}} ,
\mathcal{I}_{j_1,n}) - f_n ( \dot{\bm{X}}_{\frac{j_1-1}{n}}[i] ,
\dot{\mathcal{I}}_{j_1,n}[i])  \Big{)}\\
&\times \Big{(}f_n ( {\bm{X}}_{\frac{j_2-1}{n}} ,
\mathcal{I}_{j_2,n}) - f_n ( \dot{\bm{X}}_{\frac{j_2-1}{n}}[i] ,
\dot{\mathcal{I}}_{j_2,n}[i])  \Big{)} \Big{]} \Big|\\
I_{3}
  &:=\frac{1}{n} \sum_{j_1, j_2 = i-m}^{i + m} \Big| \E \Big{[} \Big{(}f_n ( {\bm{X}}_{\frac{j_1-1}{n}} ,
\mathcal{I}_{j_1,n}) - f_n ( \dot{\bm{X}}_{\frac{j_1-1}{n}}[i] ,
\dot{\mathcal{I}}_{j_1,n}[i])  \Big{)} \\
& \times \Big{(}f_n ( {\bm{X}}_{\frac{j_2-1}{n}} ,
\mathcal{I}_{j_2,n}) - f_n ( \dot{\bm{X}}_{\frac{j_2-1}{n}}[i] ,
\dot{\mathcal{I}}_{j_2,n}[i])  \Big{)} \Big{]} \Big|.
\end{align*}
We start studying $I_1$. As the function $f_n$ is centered, if $j_1 + m< j_2$ or $j_2  + m < j_1$ this quantity is equal to zero. Then, we have $j_1 - m \le j_2 \le j_1 + m$. Moreover we recall that, as $j \ge i +m + 1$, we have that $\dot{\mathcal{I}}_{j,n}[i] = {\mathcal{I}}_{j,n}$. Using the mean value theorem and the boundedness of the derivatives with respect to the first component of $f_n$, we thus get
\begin{align*}
|I_1| &\le \frac{C}{n} \sum_{j_1=1}^{\lfloor nt\rfloor} \sum_{j_2 = j_1 - m}^{j_1 + m} \Indi{\{j_1 \ge i+ 1\}} \Big| \E \Big{[} \Big{(}f_n ( {\bm{X}}_{\frac{j_1-1}{n}} ,
\mathcal{I}_{j_1,n}) - f_n ( \dot{\bm{X}}_{\frac{j_1-1}{n}}[i] ,
\dot{\mathcal{I}}_{j_1,n}[i])  \Big{)} \\
& \times \Big{(}f_n ( {\bm{X}}_{\frac{j_2-1}{n}} ,
\mathcal{I}_{j_2,n}) - f_n ( \dot{\bm{X}}_{\frac{j_2-1}{n}}[i] ,
\dot{\mathcal{I}}_{j_2,n}[i])  \Big{)} \Big{]} \Big| \\
& \le \frac{C}{n} \sum_{j=1}^{\lfloor nt\rfloor} \E\Big[(1 \land \|{\bm{X}}_{\frac{i + m}{n}} - {\bm{X}}_{\frac{i-1}{n}}\| )^2\Big]
\le C\E\left[(1 \land \|{\bm{X}}_{\frac{m +1}{n}}\| )^2\right].
\end{align*}
The term on the right {\rev is bounded by $Cn^{-1}(1 + \Indi{_{\{\alpha = 2\}}} \log(n))$ thanks to Equation \eqref{eq:exponewedgex1 lemma} and the logarithm can be removed when {$\bm{H}_{3}$} is satisfied because of \eqref{eq:exponewedgex2} }.\\

\noindent We now deal with $I_2$. We act differently for $j_1 \ge i+2m + 1$ and $i + m + 1 \le j_1 \le i + 2m$ {\rev and we therefore introduce $I_{2,1}$ and $I_{2,2}$ which are such that $I_2 = I_{2,1} + I_{2,2}$}. In the first case, as $j_2 + m \le i + 2m < j_1$, it is easy to check that {\rev $I_{2,1}=0$}. For $i+ m + 1 \le j_1 \le i + 2m$ the boundedness of $f_n$ is enough to obtain that {\rev $I_{2,2}$} is upper bounded by $ Cmn^{-1} \| f_n  \|^2_\infty$. Similarly, the boundedness of $f_n$ ensures
$$|I_3|   \le C n^{-1}.$$
Putting things together we obtain \eqref{eq: bound increment Z}
in general and {\rev without the logarithm}
when {$\bm{H}_{3}$} holds. The proof of the lemma is therefore concluded.}
\qed

\section*{Proof of Theorem \ref{th2}} \label{s: proof conv stable}

{Let $t_{1},\dots, t_{r}\geq0$ be a given collection of fixed times such that $t_{1}<\dots<t_r$. 
The convergence in law of $(Z_{t_1}^{(n)},\dots,Z_{t_r}^{(n)})$ towards $(\int_{0}^{t_{1}}\sqrt{A_{s}}W(ds),\dots, \int_{0}^{t_{r}}\sqrt{A_{s}}W(ds))$ is equivalent the convergence of $\mathcal{Z}_{n}:=(Z_{t_1}^{(n)},Z_{t_2}^{(n)}-Z_{t_1}^{(n)},\dots,Z_{t_r}^{(n)}-Z_{t_{r-1}}^{(n)})$ towards $(\int_{0}^{t_{1}}\sqrt{A_{s}}W(ds),\int_{t_1}^{t_{2}}\sqrt{A_{s}}W(ds), \dots, \int_{t_{r-1}}^{t_{r}}\sqrt{A_{s}}W(ds))$. The analysis of this former convergence relies} heavily on the controls in the proof of Theorem \ref{th1}. Due to Theorem \ref{stein}, the problem is reduced to proving that condition \eqref{condi} is satisfied, i.e. that the term
$$\E\left[F\,  \sum_{q=1}^r(Z_{t_q}^{(n)}-Z_{t_{q-1}}^{(n)})\frac{\partial h_{V_q}}{\partial x_{q}}(\mathcal{Z}_n)
  - F \sum_{q=1}^{r}V_q\frac{\partial^2 h_{V_q}}{\partial x_{q}^2}(\mathcal{Z}_n)\right] \rightarrow 0,$$
with 
\begin{align*}
V_q
  &:=\int_{t_{q-1}}^{t_{q}}A_{s}^2ds,
\end{align*} 	 	
converges to zero for all bounded {\rev $\mathcal{G}$-measurable} random variables $F$ and all functions {$h_{V_q} \in C^2(\mathbb{R}^r; \R)$} with bounded first and second derivatives. {As in the proof of Theorem \ref{th1}, the fundamental ingredient of the proof consists of a detailed understanding of the term 
\begin{align*}
\E\left[F\,  \sum_{q=1}^r(Z_{t_q}^{(n)}-Z_{t_{q-1}}^{(n)})\frac{\partial h_{V_q}}{\partial x_{q}}(\mathcal{Z}_n)\right],
\end{align*}
which in turn is reduced to the study of each of the individual terms $\E\left[F\, (Z_{t_q}^{(n)}-Z_{t_{q-1}}^{(n)}) \frac{\partial h_{V_q}}{\partial x_{q}}(\mathcal{Z}_n)\right]$. 
}
We observe that{
\begin{align}\label{uniformboundednessocnd}
\sup_{n\geq 1}\left \|Z_{t}^{(n)} \right \|_{L^2} \leq  \sup_{n\geq 1}\frac{1}{{\rev n}}\sum_{i=1}^{\lfloor nt\rfloor } \E[f_n^2(\bm{X}_{\frac{i-1}{n}} ,\mathcal{I}_{i,n})]^{1/2}<\infty,
\end{align}
due to the boundedness of $g$. Due to  \eqref{uniformboundednessocnd},} we can use an approximation argument to reduce the problem of showing \eqref{condi} to the case where $F$ is of the form
\begin{equation}{\label{eq: approx F}}
F := \exp{\{ \textbf{i} \textbf{v}_1 \bm{X}_{\epsilon_1}+\cdots+\textbf{i} \textbf{v}_p \bm{X}_{\epsilon_p}\}}  
\end{equation}
for some $\textbf{v}_1,\dots,\textbf{v}_{p}\in\R^d$, $\epsilon_1<\cdots< \epsilon_{p}$ and $p\in\N$. In a similar way as in the proof of Theorem \ref{th1}, let us introduce 
$$\dot{F} [i] := \exp{\{ \textbf{i} \textbf{v}_1 \dot{\bm{X}}_{\epsilon_1}[i]+\cdots+\textbf{i} \textbf{v}_p \dot{\bm{X}}_{\epsilon_p}[i]\}}.$$
We now observe {\rev that, for all $q=1,\dots, r$, }
\begin{align*}
\E\left[F\, {(Z_{t_q}^{(n)}-Z_{t_{q-1}}^{(n)})} {\frac{\partial h_{V_q}}{\partial x_{q}}}({\mathcal{Z}_n}) \right]
  & = \frac{1}{\sqrt{n}} \sum_{i={\lfloor nt_{q-1}\rfloor+1}}^{{\lfloor nt_q\rfloor }} \E \left[f_n(\bm{X}_{\frac{i-1}{n}} ,\mathcal{I}_{i,n}) F {\frac{\partial h_{V_q}}{\partial x_{q}}}({\mathcal{Z}_n})\right]. 
\end{align*}
{Consider perturbed vector
\begin{align*}
\dot{\mathcal{Z}}_{n}[i]
  &:=(Z_{t_1}^{(n)}[i],Z_{t_2}^{(n)}[i]-Z_{t_1}^{(n)}[i],\dots,Z_{t_r}^{(n)}[i]-Z_{t_{r-1}}^{(n)}[i]).
\end{align*}
Then, we can write 
\begin{multline*}
\E\left[F\, {(Z_{t_q}^{(n)}-Z_{t_{q-1}}^{(n)})} {\frac{\partial h_{V_q}}{\partial x_{q}}}({\mathcal{Z}_n}) \right]\\
\begin{aligned}
& = \frac{1}{\sqrt{n}} \sum_{i={\lfloor nt_{q-1}\rfloor+1}}^{{\lfloor nt_q\rfloor }} \E \left[f_n(\bm{X}_{\frac{i-1}{n}} ,\mathcal{I}_{i,n}) (F - \dot{F}[i]) {\frac{\partial h_{V_q}}{\partial x_{q}}}({\mathcal{Z}_n})\right]  \\
& + \frac{1}{\sqrt{n}} \sum_{i={\lfloor nt_{q-1}\rfloor+1}}^{{\lfloor nt_q\rfloor }} \E \left[f_n(\bm{X}_{\frac{i-1}{n}} ,\mathcal{I}_{i,n}) \dot{F}[i] ({\frac{\partial h_{V_q}}{\partial x_{q}}}({\mathcal{Z}_n})- {\frac{\partial h_{V_q}}{\partial x_{q}}}({\dot{\mathcal{Z}}_n[i]}) \right] \\
& + \frac{1}{\sqrt{n}} \sum_{i={\lfloor nt_{q-1}\rfloor+1}}^{{\lfloor nt_q\rfloor }} \E \left[f_n(\bm{X}_{\frac{i-1}{n}} ,\mathcal{I}_{i,n}) \dot{F}[i] {\frac{\partial h_{V_q}}{\partial x_{q}}}({\dot{\mathcal{Z}}_n[i]})\right]\\
& = : I_1 + I_2 + I_3.
\end{aligned}
\end{multline*}
}

We start dealing with $I_3$. We observe that $\dot{F}[i] {\frac{\partial h_{V_q}}{\partial x_{q}}}({\dot{\mathcal{Z}}_n[i]})$ is measurable with respect to $\mathcal{F}^{(i)}$. Then, 
\begin{align*}
I_3 = \frac{1}{\sqrt{n}} \sum_{i={\lfloor nt_{q-1}\rfloor+1}}^{{\lfloor nt_q\rfloor }} \E \left[ \E[f_n(\bm{X}_{\frac{i-1}{n}} ,\mathcal{I}_{i,n})| \mathcal{F}^{(i)}] \dot{F}[i] {\frac{\partial h_{V_q}}{\partial x_{q}}}({\dot{\mathcal{Z}}_n[i]})\right] = 0.
\end{align*}
Regarding $I_1$, from the boundedness of $f_n$ and ${\nabla h_{V_q}}$ we conclude that
$$|I_1| \le \frac{C}{\sqrt{n}} \sum_{i=1}^{\lfloor nt\rfloor } \E \left[|F - \dot{F}[i]|\right],$$
{where $t$ is defined by $t:=t_1\vee\cdots\vee t_{r}$.}  Thus, from the definition of $F$ and $\dot{F}[i]$ it is straightforward to see that, {\rev thanks to the fact that $\bm{X}_{s} - \dot{\bm{X}}_{s}[i]$ is as in \eqref{eq:1new},}
\begin{align}{\label{eq: bound F F(i)}}
\E \left[|F - \dot{F}[i]|\right] 
& \le C n^{- \frac{1}{\alpha}} + Cn^{-1}\left(1 + \Indi{ \{ \alpha = 1 \}} \log n\right), 
\end{align}
where we have also used Lemma \ref{eq:ineqoneminx}. It follows that
$$|I_1| \le C n^{\frac{1}{2}- \frac{1}{\alpha}} + Cn^{-\frac 12}\left(1 + \Indi{ \{ \alpha = 1 \}} \log n\right),$$
which clearly converges to $0$ for $n \rightarrow \infty$. \\

\noindent To conclude the proof we deal with $I_2$. {\revch We underline that it is the same as the main term appearing in \eqref{eq:SnIBPintro} in the previous proof, with an extra $\dot{F}[i]$, before the application of Taylor approximation. It is easy to check that, as $\dot{F}[i]$ is independent from $\mathcal{F}^{(i)}$, the proof of Theorem \ref{th1} starting from Taylor development as in \eqref{eq: Taylor} still applies.} { \rev Define 
\begin{align*}
I_{2,0}
  &:=\frac{1}{\sqrt{n}} \sum_{p=1}^r\Indi{\{p\neq q\}}\sum_{i=\lfloor nt_{q-1}+1\rfloor}^{\lfloor nt_q\rfloor } \E \left[f_n(\bm{X}_{\frac{i-1}{n}} ,\mathcal{I}_{i,n}) \dot{F}[i] \frac{\partial^2 h_{V_{q}}}{\partial x_{q}x_{p}}(\mathcal{Z}_n)(Z_{t_p}^{(n)} - \dot{Z}_{t_p}^{(n)}[i]) \right]\\
I_{2,1}
  &:=\frac{1}{\sqrt{n}} \sum_{i=\lfloor nt_{q-1}+1\rfloor}^{\lfloor nt_q\rfloor } \E \left[f_n(\bm{X}_{\frac{i-1}{n}} ,\mathcal{I}_{i,n}) \dot{F}[i] \frac{\partial^2 h_{V_{q}}}{\partial x_{q}^2}({\mathcal{Z}_n})(Z_{t_q}^{(n)} - \dot{Z}_{t_q}^{(n)}[i]) \right],
\end{align*}
and $I_{2,2}:=I_2 -(I_{2,0}+I_{2,1})$. {\revch Observe that $I_{2,2}$ plays a similar role to $T_{3,n}$ in the proof of Theorem \ref{th1}. However, on $I_{2,0} + I_{2,1}$, we have not yet applied the decomposition derived from \eqref{eq: Sn 1.25} and \eqref{eq: 5.13.5}. Once applied, it will become evident that both $I_{2,0}$ and $I_{2,1}$ can be viewed as $T_{1,n} + T_{2,n} + T_{4,n}$, where $h''_V(Z_t^{(n)})$ is replaced by $\frac{\partial^2 h_{V_{q}}}{\partial x_{q}x_{p}}(\mathcal{Z}n)$ in the first case and $\frac{\partial^2 h_{V_{q}}}{\partial x_{q}^2}({\mathcal{Z}_n})$ in the second. Then, as} in the proof of Theorem  \ref{th1}, we have that $I_{2,2}$ goes to $0$ due to the $L^2$ bound on the increments $(Z_{t}^{(n)} - \dot{Z}_{t}^{(n)}[i])$ gathered in Lemma \ref{lemma: L2 z-zdot}, while $I_{2,0}+I_{2,1}$ is the main contribution. {\revch Applying the same decomposition and the same approach as in the proof of Theorem \ref{th1}, and using that} $\beta_n$ defined as in \eqref{eq: def beta} converges to $0$ because of \eqref{eq: var to zero}, one obtains that $|I_{2,1} - \E[\dot{F}[i] V_q \frac{\partial^2 h_{V_q}}{\partial x_q^2}(\mathcal{Z}_n) ]| \rightarrow 0$ and $I_{2,0} \rightarrow 0$ for $n \rightarrow \infty$. \\
The proof is complete once one notices that, according with the arguing above, 
$$I_2 - \E[F V_q \frac{\partial^2 h_{V_q}}{\partial x_q^2}(\mathcal{Z}_{n}) ] = \E[( F -\dot{F}[i]) V_q \frac{\partial^2 h_{V_q}}{\partial x_q^2}(\mathcal{Z}_{n}) ] + o(1)$$
and that the first term in the right hand side above converges towards $0$ for $n \rightarrow \infty$ thanks to \eqref{eq: bound F F(i)}. }
\qed

\section{Technical lemmas} \label{s: technical lemmas}
\begin{lemma}\label{lem:distancegaussians}
Let $X\sim \mathcal N(0,\sigma^2)$ and $Y\sim \mathcal N(0,s^2)$. Then we have that for every function $\psi\in C^4(\R;\R)$  with $\|\psi^{(4)}\|_{\infty}<\infty$, it holds that 
\begin{align*}
|\E[\psi(X)-\psi(Y)-(\sigma^2-s^2)\psi^{\prime\prime}(Y)]|
  &\leq 4\|\psi^{(4)}\|_{\infty}|s^{2}-\sigma^2|^{2}.
\end{align*}
\end{lemma}

\begin{proof}
Suppose first that $\sigma^2>s^2$. {\rev Remark that in the left hand side above we can replace $X$ and $Y$ with any coupling of $X$ and $Y$. In particular, we can choose $X$ and $Y$ such that}, $X\stackrel{Law}{=}Y+\varepsilon Z$, where $\varepsilon:=\sqrt{\sigma^2-s^2}$ and $Z$ is a standard Gaussian random variable independent of $Y$. By the independence of $Y$ and $Z$, $\E[\psi^{\prime}(Y)Z]=\E[\psi^{\prime}(Y)Z^3]=0$ and $\E[\psi^{\prime}(Y)Z^2]=\E[\psi^{\prime}(Y)]$. Consequently, 
\begin{multline*}
|\E[\psi(X)-\psi(Y)-(\sigma^{2}-s^2)\psi^{\prime\prime}(Y)]|\\
\begin{aligned}
  &=|\E[\psi(Y+\varepsilon Z)-\psi(Y)-\psi^{\prime}(Y)(\varepsilon Z)-\psi^{\prime\prime}(Y)(\varepsilon^2Z^2)-\psi^{\prime\prime\prime}(Y)(\varepsilon^3Z^3)]]|\\
	&\leq \|\psi^{(4)}\|_{\infty}\E[|\varepsilon Z|^4]|=3\|\psi^{(4)}\|_{\infty}\varepsilon^4.
\end{aligned}
\end{multline*}
In the case $\sigma^2<s^2$, we have that $Y\stackrel{Law}{=}X+\delta Z$, where $\delta:=\sqrt{s^2-\sigma^2}$. Proceeding as before, we get
\begin{align*}
|\E[\psi(Y)-\psi(X)-(s^{2}-\sigma^2)\psi^{\prime\prime}(X)]| 
	&\leq  3\|\psi^{(4)}\|_{\infty}\delta^4.
\end{align*}
Combining this inequality with the relation
\begin{align*}
|\E[(s^{2}-\sigma^2)\psi^{\prime\prime}(X)]-\E[(s^{2}-\sigma^2)\psi^{\prime\prime}(Y)]|
  &\leq \delta^4 \|\psi^{(4)}\|_{\infty},
\end{align*} 
we get 
\begin{align*}
|\E[\psi(Y)-\psi(X)-(s^{2}-\sigma^2)\psi^{\prime\prime}(Y)]| 
	&\leq  4\|\psi^{(4)}\|_{\infty}\delta^4.
\end{align*}
The result easily follows from here.
\end{proof}

\begin{lemma}\label{eq:ineqoneminx}
Let $\bm{X}$ be a process satisfying \eqref{eq:Xtails} {for some $\alpha\in(0,2]$}. Then 
\begin{align}\label{eq:exponewedgex1 lemma}
\E\left[1\wedge\|\bm{X}_{1/n}\|^2\right]
  &\leq Cn^{-1}\left(1+\Indi{\{\alpha=2\}}\log(n)\right).
\end{align}
and 
\begin{align}\label{eq:exponewedgexv2}
\E\left[1\wedge\|\bm{X}_{1/n}\|\right]
  &\leq C \left(n^{-\frac{1}{\alpha}}+n^{-1}\left(1+\Indi{\{\alpha=1\}}\log(n)\right)\right).
\end{align}
{\rev In the case $\alpha>1$, the above inequalities generalize to 
\begin{align}\label{eq:exponewedgex1wemma}
\E\left[\|\bm{X}_{1/n}\|\wedge\|\bm{X}_{1/n}\|^2\right]
  &\leq Cn^{-1}\left(1+\Indi{\{\alpha=2\}}\log(n)\right).
\end{align}
If $\alpha=2$ and $\textbf{H}_3$ holds true for all $t\leq 1$, then 
\begin{align}\label{eq:exponewedgex2}
\E\left[\|\bm{X}_{1/n}\|\wedge\|\bm{X}_{1/n}\|^2\right]
  &\leq Cn^{-1}.
\end{align}}

\end{lemma}
\begin{proof}
{We first consider the case $\alpha\in(0,2)$. Observe that 
\begin{align}\label{eq:EonewedgeXbound}
\E[1\wedge\|\bm{X}_{1/n}\|^2]
  &=\Pb[\|\bm{X}_{1/n}\|\geq 1]+\E[\Indi{\{\|\bm{X}_{1/n}\|\leq 1\}}\|\bm{X}_{1/n}\|^2]
\end{align}
so that 
\begin{align*}
\E[1\wedge\|\bm{X}_{1/n}\|^2]
  &\leq \kappa n^{-1}+\E[\Indi{\{\|\bm{X}_{1/n}\|\leq 1\}}\|\bm{X}_{1/n}\|^2]\\
	&= \kappa n^{-1}+2\E\left[\int_{\R_{+}}s\Indi{\{s\leq \|\bm{X}_{1/n}\|\leq 1\}}ds\right]
	= \kappa n^{-1}+2\int_{0}^1s\Pb[s\leq \|\bm{X}_{1/n}\|]ds\\
	&\leq \kappa n^{-1}+2\int_{0}^1s(1\wedge(\kappa n^{-1}s^{-\alpha}))ds.
\end{align*}
}We get
\begin{align}\label{ineq:integralappendixbound}
\int_{0}^1s(1\wedge(\kappa n^{-1}s^{-\alpha}))ds
	&\leq \kappa n^{-1}\int_{\kappa^{\frac{1}{\alpha}} n^{-\frac{1}{\alpha}}}^1s^{1-\alpha}ds
	+\int_{0}^{\kappa^{\frac{1}{\alpha}} n^{-\frac{1}{\alpha}}}sds\nonumber\\
	&
	{\rev \leq C\left( n^{-1} +  n^{- \frac{2}{\alpha}}\right) \le C n^{-1},} 
\end{align}
yielding the desired result. For handling the case $\alpha=2$, the above analysis gives
\begin{align}{\label{eq: 6.75 bound int}}
\int_{0}^1s(1\wedge(\kappa n^{-1}s^{-2}))ds
	&\leq n^{-1} \log(\kappa^{-\frac{1}{2}} n^{\frac{1}{2}})
	+\frac{1}{2} \kappa n^{-1}, 
\end{align}
giving the desired bound. The proof of \eqref{eq:exponewedgexv2} follows by a similar reasoning. Indeed, 
\begin{align*}
\E[1\wedge\|\bm{X}_{1/n}\|]
  &=\Pb[\|\bm{X}_{1/n}\|\geq 1]+\E[\Indi{\{\|\bm{X}_{1/n}\|\leq 1\}}\|\bm{X}_{1/n}\|] \\
	& \le \kappa n^{-1}+\E\left[\int_{\R_{+}} \Indi{\{s\leq \|\bm{X}_{1/n}\|\leq 1\}}ds\right] \leq \kappa n^{-1}+\int_{0}^1(1\wedge(\kappa n^{-1}s^{-\alpha}))ds.
\end{align*}
As before we obtain for $\alpha \neq 1$:
$$\int_{0}^1\left(1\wedge(\kappa n^{-1}s^{-\alpha})\right)ds  \le \kappa n^{-1} + \kappa^{\frac{1}{\alpha}} n^{-\frac{1}{\alpha}},$$
while for $\alpha = 1$ we have 
$$\int_{0}^1\left(1\wedge(\kappa n^{-1}s^{-1})\right)ds  \le \kappa n^{-1} \log(n) + \kappa^{\frac{1}{\alpha}} n^{-1}.$$
Hence, \eqref{eq:exponewedgexv2} holds true.

{\noindent For proving inequality \eqref{eq:exponewedgex1wemma}, we observe that 
\begin{align*}
\E[\|\bm{X}_{1/n}\|\wedge\|\bm{X}_{1/n}\|^2]
  &=\E[\Indi{\{\|\bm{X}_{1/n}\|\geq 1\}}\|\bm{X}_{1/n}\|]+\E[\Indi{\{\|\bm{X}_{1/n}\|\leq 1\}}\|\bm{X}_{1/n}\|^2]
\end{align*}
so that 
{
\begin{align*}
\E[\|\bm{X}_{1/n}\|\wedge\|\bm{X}_{1/n}\|^2]
	&= \E\left[\int_{\R_{+}}\Indi{\{1\vee s\leq \|\bm{X}_{1/n}\|\}}ds\right]+2\E\left[\int_{\R_{+}}s\Indi{\{s\leq \|\bm{X}_{1/n}\|\leq 1\}}ds\right]\\
	&\leq \int_{0}^{\infty} \kappa n^{-1}(1\vee s)^{-\alpha} ds+2\int_{0}^1s(1\wedge(\kappa n^{-1}s^{-\alpha}))ds.
\end{align*}}
The first integral is bounded by $C/n$ due to the condition $\alpha>1$, while the second one can be bounded by means of inequality \eqref{ineq:integralappendixbound}, {\rev while the case $\alpha = 2$ is studied in \eqref{eq: 6.75 bound int}}. Inequality \eqref{eq:exponewedgex1wemma} follows from here.
}

\noindent For the case where $\alpha=2$ and in addition {\rev $\textbf{H}_3$ holds for some $\gamma>1$}, then 
{\rev
\begin{align*}
 \E[\|\bm{X}_{1/n}\|\wedge\|\bm{X}_{1/n}\|^2]
  &=
  \E[\Indi{\{\|\bm{X}_{1/n}\|\leq 1\}}\|\bm{X}_{1/n}\|^2] + \E[\Indi{\{\|\bm{X}_{1/n}\|\geq 1\}}\|\bm{X}_{1/n}\|].
  \end{align*}
Now, using that 
$$\int_0^{\|\bm{X}_{1/n}\|} s ds = \frac{1}{2} \|\bm{X}_{1/n}\|^2, \qquad \int_0^{\|\bm{X}_{1/n}\|} ds = \|\bm{X}_{1/n}\| $$
we have

  \begin{align*}
\E[\|\bm{X}_{1/n}\|\wedge\|\bm{X}_{1/n}\|^2]
	&\leq 
 2\E\left[\int_{\R_{+}}s\Indi{\{s\leq \|\bm{X}_{1/n}\|\leq 1\}}ds\right] + \E\left[\int_{\R_{+}}\Indi{\{s \lor 1 \leq \|\bm{X}_{1/n}\| \}}ds\right]  \\
	&= 
 2 \int_{0}^1s\Pb[s\leq \|\bm{X}_{1/n}\|]ds + \int_{0}^\infty \Pb[\|\bm{X}_{1/n}\| \geq 1 \lor s]ds\\
	&\leq n^{-1}+
 2\int_{0}^1s(1\wedge(\kappa (n^{-1}s^{-2})^{\gamma}))ds + {\kappa} \int_0^\infty \Big(n^{-1}(1 \lor s)^{-2} \Big)^\gamma ds.
\end{align*}
The first integral on the right can be bounded as follows 
\begin{align*}
\int_{0}^1s(1\wedge(\kappa (n^{-1}s^{-2})^{\gamma}))ds
  &\leq n^{-1}\int_{0}^{\infty}u(1\wedge(\kappa u^{-2\gamma}))du,
\end{align*}
while the second is bounded by $n^{- \gamma}(\int_0^1 1 ds + \int_1^\infty s^{-2 \gamma} ds)$.
The result follows from the fact that $\gamma>1$. }

\end{proof}

\bibliographystyle{plain}
\bibliography{bibliography}

\begin{thebibliography}{10}

\bibitem{AJP22}
Chiara Amorino, Arturo Jaramillo, and Mark Podolskij.
\newblock Optimal estimation of the local time and the occupation time measure for an $\alpha$-stable l{\'e}vy process.
\newblock {\em Modern Stochastics: Theory and Applications}, pages 1--20, 2024.

\bibitem{Ber96}
Jean Bertoin.
\newblock {\em L\'{e}vy processes}, volume 121 of {\em Cambridge Tracts in Mathematics}.
\newblock Cambridge University Press, Cambridge, 1996.

\bibitem{CheGoShao}
Louis H.~Y. Chen, Larry Goldstein, and Qi-Man Shao.
\newblock {\em Normal Approximation by {S}tein's Method}.
\newblock Probability and its Applications (New York). Springer, Heidelberg, 2011.

\bibitem{HuNuXu}
Yaozhong Hu, David Nualart, and Fangjun Xu.
\newblock Central limit theorem for an additive functional of the fractional {B}rownian motion.
\newblock {\em Ann. Probab.}, 42(1):168--203, 2014.

\bibitem{I18}
Jevgenijs Ivanovs.
\newblock Zooming in on a {L}\'{e}vy process at its supremum.
\newblock {\em Ann. Appl. Probab.}, 28(2):912--940, 2018.

\bibitem{J97}
Jean Jacod.
\newblock On continuous conditional {G}aussian martingales and stable convergence in law.
\newblock In {\em S\'{e}minaire de {P}robabilit\'{e}s, {XXXI}}, volume 1655 of {\em Lecture Notes in Math.}, pages 232--246. Springer, Berlin, 1997.

\bibitem{JP12}
Jean Jacod and Philip Protter.
\newblock {\em Discretization of Processes}, volume~67 of {\em Stochastic Modelling and Applied Probability}.
\newblock Springer, Heidelberg, 2012.

\bibitem{MR3495691}
Jean Jacod and Viktor Todorov.
\newblock Efficient estimation of integrated volatility in presence of infinite variation jumps with multiple activity indices.
\newblock In {\em The fascination of probability, statistics and their applications}, pages 317--341. Springer, Cham, 2016.

\bibitem{JaNoNuPE}
Arturo Jaramillo, Ivan Nourdin, and Giovanni Peccati.
\newblock Limit theorems for additive functionals of the fractional brownian motion.
\newblock {\em Ann. Appl. Probab.}, 2022.

\bibitem{Kyp06}
Andreas~E. Kyprianou.
\newblock {\em Introductory lectures on fluctuations of {L}\'{e}vy processes with applications}.
\newblock Universitext. Springer-Verlag, Berlin, 2006.

\bibitem{Lindeberg}
J.~W. Lindeberg.
\newblock Eine neue {H}erleitung des {E}xponentialgesetzes in der {W}ahrscheinlichkeitsrechnung.
\newblock {\em Math. Z.}, 15(1):211--225, 1922.

\bibitem{NoNu}
Ivan Nourdin and David Nualart.
\newblock Central limit theorems for multiple {S}korokhod integrals.
\newblock {\em J. Theoret. Probab.}, 23(1):39--64, 2010.

\bibitem{NoNuPe}
Ivan Nourdin, David Nualart, and Giovanni Peccati.
\newblock Quantitative stable limit theorems on the {W}iener space.
\newblock {\em Ann. Probab.}, 44(1):1--41, 2016.

\bibitem{NP12}
Ivan Nourdin and Giovanni Peccati.
\newblock {\em Normal Approximations with {M}alliavin Calculus}, volume 192 of {\em Cambridge Tracts in Mathematics}.
\newblock Cambridge University Press, Cambridge, 2012.
\newblock From Stein's method to universality.

\bibitem{NR09}
Ivan Nourdin and Anthony R\'{e}veillac.
\newblock Asymptotic behavior of weighted quadratic variations of fractional {B}rownian motion: the critical case {$H=1/4$}.
\newblock {\em Ann. Probab.}, 37(6):2200--2230, 2009.

\bibitem{NuXu}
David Nualart and Fangjun Xu.
\newblock Central limit theorem for an additive functional of the fractional {B}rownian motion {II}.
\newblock {\em Electron. Commun. Probab.}, 18:no. 74, 10, 2013.

\bibitem{NY19}
David Nualart and Nakahiro Yoshida.
\newblock Asymptotic expansion of {S}korohod integrals.
\newblock {\em Electron. J. Probab.}, 24:Paper No. 119, 64, 2019.

\bibitem{PVY17}
Mark Podolskij, Bezirgen Veliyev, and Nakahiro Yoshida.
\newblock Edgeworth expansion for the pre-averaging estimator.
\newblock {\em Stochastic Process. Appl.}, 127(11):3558--3595, 2017.

\bibitem{PVY20}
Mark Podolskij, Bezirgen Veliyev, and Nakahiro Yoshida.
\newblock Edgeworth expansion for {E}uler approximation of continuous diffusion processes.
\newblock {\em Ann. Appl. Probab.}, 30(4):1971--2003, 2020.

\bibitem{PY16}
Mark Podolskij and Nakahiro Yoshida.
\newblock Edgeworth expansion for functionals of continuous diffusion processes.
\newblock {\em Ann. Appl. Probab.}, 26(6):3415--3455, 2016.

\bibitem{R63}
Alfr\'{e}d R\'{e}nyi.
\newblock On stable sequences of events.
\newblock {\em Sankhy\={a} Ser. A}, 25:293 302, 1963.

\bibitem{RY}
Daniel Revuz and Marc Yor.
\newblock {\em Continuous martingales and {B}rownian motion}, volume 293 of {\em Grundlehren der mathematischen Wissenschaften [Fundamental Principles of Mathematical Sciences]}.
\newblock Springer-Verlag, Berlin, third edition, 1999.

\bibitem{Sat13}
Ken-iti Sato.
\newblock {\em L\'{e}vy processes and infinitely divisible distributions}, volume~68 of {\em Cambridge Studies in Advanced Mathematics}.
\newblock Cambridge University Press, Cambridge, 2013.
\newblock Translated from the 1990 Japanese original, Revised edition of the 1999 English translation.

\bibitem{ChStein}
Charles Stein.
\newblock A bound for the error in the normal approximation to the distribution of a sum of dependent random variables.
\newblock In {\em Proceedings of the {S}ixth {B}erkeley {S}ymposium on {M}athematical {S}tatistics and {P}robability ({U}niv. {C}alifornia, {B}erkeley, {C}alif., 1970/1971), {V}ol. {II}: {P}robability theory}, pages 583--602. Univ. California Press, Berkeley, Calif., 1972.

\bibitem{MR3226166}
Viktor Todorov and George Tauchen.
\newblock Limit theorems for the empirical distribution function of scaled increments of {I}t\^{o} semimartingales at high frequencies.
\newblock {\em Ann. Appl. Probab.}, 24(5):1850--1888, 2014.

\bibitem{Y13}
Nakahiro Yoshida.
\newblock Martingale expansion in mixed normal limit.
\newblock {\em Stochastic Process. Appl.}, 123(3):887--933, 2013.

\end{thebibliography}
\end{document}